\documentclass[10pt,twoside]{article}
\usepackage{graphics,color}
\usepackage{graphicx}
\usepackage{epstopdf}
\usepackage{amsfonts,amsmath,amsthm,amssymb,bm}
\usepackage{enumerate}
\usepackage[pdftex, colorlinks=true, citecolor=green]{hyperref}
\usepackage{lscape}

\newcommand{\ds}{\displaystyle}

\newtheorem{thm}{Theorem}[section]

\newcommand{\ttau}{\tilde{\tau}}
\newcommand{\otau}{\overline{\tau}}

\begin{document}

\title{Stability of coupled Wilson-Cowan systems with distributed delays}

\author{Eva Kaslik$^{1}$ \and Emanuel-Attila K\"ok\"ovics$^1$ \and
Anca R\u adulescu$^{2,*}$
}

\date{\noindent $^1$ West University of Timi\c{s}oara, Bd.  V. P\^{a}rvan nr. 4, 300223, Timi\c{s}oara, Romania\\
\noindent $^2$ State University of New York at New Paltz, New Paltz, NY 12561, USA\\ 
$^*$ Corresponding Author: {\it radulesa@newpaltz.edu}}

\maketitle

\begin{abstract}
\noindent Building upon our previous work on the Wilson-Cowan equations with distributed delays, we study the dynamic behavior in a system of two coupled Wilson-Cowan pairs. We focus in particular on understanding the mechanisms that govern the transitions in and out of oscillatory regimes associated with pathological behavior. We investigate these mechanisms under multiple coupling scenarios, and we compare the effects of using discrete delays versus a weak Gamma delay distribution. We found that, in order to trigger and stop oscillations, each kernel emphasizes different critical combinations of coupling weights and time delay, with the weak Gamma kernel restricting oscillations to a tighter locus of coupling strengths, and to a limited range of time delays. We finally illustrate the general analytical results with simulations for two particular applications: generation of beta-rhytms in the basal ganglia, and alpha oscillations in the prefrontal-limbic system. 
\end{abstract}

\section{Introduction}

The Wilson-Cowan system~\cite{wilson1972excitatory} is perhaps the best known and most studied mean field model in coupled excitatory and inhibitory neural populations. The pair of nonlinear ordinary differential equations broadly known as the Wilson-Cowan (WC) model provide a beautiful description of the rich behavior that emerges when observing firing patterns of two interacting excitatory and inhibitory neural populations. The variables of the system $E(t)$ and $I(t)$ are the proportion of excitable cells at time $t$. The model encompasses in an efficient, simplified form a variety of factors crucial to these firing patterns: external inputs $P_e(t)$ and $P_i(t)$, coupling strengths $c_i$, $1 \leq i \leq 4$ within and between the excitatory and inhibitory populations, as well as the critical time constant for neural integration $\otau$ and the history of firing in the cells.  The WC nonlinear model was originally conceived as a system of integral equations:
\begin{equation}
\left\{\begin{array}{l}
\ds E(t+\tau_1)=\left(1-\int_{t-r}^tE(s)ds\right)\cdot\mathcal{S}_e\left[\int_{-\infty}^t h(t-s)\left(c_1E(s)-c_2I(s)+P_e(s)\right)ds\right]\\
\ds I(t+\tau_2)=\left(1-\int_{t-r'}^tI(s)ds\right)\cdot\mathcal{S}_i\left[\int_{-\infty}^t h(t-s)\left(c_3E(s)-c_4I(s)+P_i(s)\right)ds\right]
\end{array}\right.
\label{unsimplified}
\end{equation}
The cells which are actively firing at time $t+\otau$ are the cells that receive sufficient input to raise them to above-threshold potential, provided the cells are not in refractory state. The first factor on the right insures that only cells that have transcended the refractory period are considered, where $r$ is the length of this period in msecs. The mean field excitatory / inhibitory inputs received by each population are integrated through the nonlinear, sigmoidal functions $\mathcal{S}_e$ and respectively $\mathcal{S}_i$, to determine the effective drive. To realistically model the biology, is was originally assumed that in each case the sum of inputs is modulated in time by a decreasing kernel $h(t)$, accounting for the memory of firing past firing, which informs on the likelihood for the cells to fire again. The use of a weak Gamma kernel $h(t)$ is briefly suggested in the original paper~\cite{wilson1972excitatory}, before applying a course graining method and simplifying the system, for the convenience of the analysis, to a pair of coupled nonlinear ODEs:
\begin{equation}
\left\{\begin{array}{l}
\otau \dot{\hat{E}}(t)= -\hat{E}(t) + (1-r\hat{E}) \mathcal{S}_e\left(kc_1\hat{E}(s)-kc_2\hat{I}(s)+kP(s)\right)\\
\otau \dot{\hat{I}}(t)= -\hat{I}(t) +(1-r'\hat{I})\mathcal{S}_i\left(k'c_3\hat{E}(s)-k'c_4\hat{I}(s)+k'Q(s)\right)
\end{array} \right.
\end{equation}
where $\ds \int_{t-r}^t E(s) \:ds$ and $\ds \int_{t-r}^t I(s) \:ds$ 
 were approximated to $r\hat{E}(t)$ and $r'\hat{I}(t)$, while $\ds \int_{-\infty}^t h(t-s) E(s) \:ds$ and $\ds \int_{-\infty}^t h(t-s) I(s) \:ds$ were approximated by $k\hat{E}(t)$ and $k' \hat{E}(t)$, respectively. Although the final step of course-graining is common practice in modeling neural function, it can result in the loss of crucial information about temporal integration of synaptic inputs, since it assumes that the temporal input function $h(t)$ is close to one for $t < r, r'$ and decays instantly to zero for $t > r, r'$. Therefore, in some variations of the standard Wilson-Cowan model, the authors found it beneficial to restore this important aspect by examining the equations before the course-graining, typically assuming discrete time delays. To maintain tractability of the resulting integro-differential system, this is typically done in conjunction with assuming null refractory periods $r,r'$, resulting in the system:

\begin{equation}
\left\{\begin{array}{l}
\otau \dot{\hat{E}}(t)= -\hat{E}(t) + \mathcal{S}_e\left(kc_1\hat{E}(s)-kc_2\hat{I}(s)+kP(s)\right)\\
\otau \dot{\hat{I}}(t)= -\hat{I}(t) + \mathcal{S}_i\left(k'c_3\hat{E}(s)-k'c_4\hat{I}(s)+k'Q(s)\right)
\end{array} \right.
\end{equation}

While in our own prior work we also assumed $r,r'=0$, we aimed to obtain more general, kernel-independent result. With these results, we can gain a more comprehensive understanding of the potential of using time delays to model more realistically aspects like non-instantaneous synaptic transmission and time-course of post-synaptic potentials. General results also allowed us to compare dynamic patterns between delay distributions, and to generate hypotheses related to which kernel can be optimally used by a neural population designed to accomplish specific dynamic transitions and functional behaviors.

 In this paper, we extend our work to study dynamics in a system of two coupled Wilson-Cowan systems with distributed delays. This is following in the path of existing studies of dynamics in a pair of  symmetrically or non-symmetrically coupled WC systems. For example work by Borisyuk and others~\cite{borisyuk1995dynamics,ueta2003synchronization,neves2016linear} has explored dynamic patterns in coupled nonlinear systems (without delays), leading to numerical descriptions of rich transitions from stable equilibria, symmetric, anti-symmetric and non-symmetric cycles, to invariant tori and chaos. Dynamics of a coupled pair of WC systems had also been considered by Wang et al~\cite{wang2023possible}, in the case of discrete delays and for a specific, application-targeted circuit.

 In our work, we encompass a more general class of coupling schemes, inspired by realistic neuro-regulatory circuits, but which are also simple enough for formal analyses and efficient simulations. We first generate kernel-independent results on stability and bifurcations in the system. We then illustrate how these transpire differently in two specific cases: the Dirac kernel (used in most existing applications) and the weak Gamma kernel (proposed by Wilson and Cowan in their original reference~\cite{wilson1972excitatory}). 

 Finally, we illustrate our approach on two specific brain circuits. First, we revisit and improve on an existing feedback circuit studied by Wand et al.~\cite{wang2023possible}, modeling the  involvement of the cortex to basal ganglia coupling in the firing behavior observed in Parkinson's Disease~\cite{little2014functional}). Second, we illustrate how our general analytical framework can be used to understand the dynamics of the prefronal-amygdala coupling, and propose potential mechanisms for the gamma rhythms associated with  emotion-regulation~\cite{headley2021gamma}.
 
To begin with, we investigate four simple coupling schemes between the two Wilson-Cowan systems. These are a limited collection when comparing with real brain circuitry, which is much more complex. In the applications described here, we use these schemes to model local loops extracted from ampler and much more complex global connectivity schemes. This is only a first step along the lines of including the network aspect in conjunction with the temporal aspect encapsulate in the distributed delays. It can be viewed as a proof of principle that analytical methods employed in a two-dimensional system could be easily adapted to obtain useful general results for coupling of two such systems, and to compare behaviors between different coupling schemes. Our future work is centered around exploring how to generalize such results in more complex, higher-dimensional networks.

\section{Mathematical model}

We consider two pairs of excitatory / inhibitory Wilson-Cowan populations (denoted $E_1$ / $I_1$ and $E_2$ / $I_2$, respectively). The dynamics of the four coupled units are described by the following system
\begin{equation}\label{sys.nelin}
\ds \overline{\tau}\dot{X_j} = -X_j(t) + {\cal F}_j\left( \int_{-\infty}^t [C_j] X(s) h(t-s) \: ds + P_j \right), \quad 1 \leq j \leq 4
\end{equation}

where the four-dimensional variable is the time-dependent column vector $X(t) = (E_1(t),I_1(t),E_2(t),I_2(t))^T$; $\overline{\tau}$ is the system's time constant (assumed identical in all components for simplicity); $C$ is the $4 \times 4$ matrix of connection weights, with $[C_j]$ designating its $j$th row; $P_j$ is the external input in each component; ${\cal F}_j$ is the integrating sigmoidal function for each variable.

An equilibrium point $\mathcal{E}^\star = (X_1^\star, X_2^\star, X_3^\star, X_4^\star)^T$ of non-linear system \eqref{sys.nelin} has to satisfy the algebraic system:
\begin{equation*}\label{sys.algebraic}
\ds -X_j^\star + {\cal F}_j \left( [C_j] {\cal E}^* +P_j \right) = 0
\end{equation*}

The local behavior of the system \eqref{sys.nelin} around such an equilibirum is captured by its linearization at $\mathcal{E}^*$:

\begin{equation*}\label{sys.lin}
\ds  \dot{X_j} = -\dfrac{1}{\overline{\tau}}X_j(t) + \dfrac{1}{\overline{\tau}} \phi_j\left( \int\limits_{-\infty}^t [C_j]X(s)h(t-s) \: ds \right)
\end{equation*}

where we denote
\begin{equation*}
\ds \phi_j = \phi_j(\mathcal{E}^\star) = {\cal F}'_j\left( [C_j]{\cal E}^* + P_j \right)
\end{equation*}

Looking for solutions of the form $X(t) = e^{zt}X(0)$ or employing the method of Laplace transform, we obtain the following:
\begin{equation}
\ds  z\mathcal{L}_j - X_j(0) = -\dfrac{1}{\overline{\tau}}\mathcal{L}_j + \dfrac{1}{\overline{\tau}} \phi_j\left( H(z)[C_j]{\cal L} \right)
\end{equation}
where ${\cal L}_j(z)$ is the Laplace transform of $X_j(t)$, $H(z)$ is the transform of the kernel $h(t)$, and the column vector ${\cal L} = ({\cal L}_1,{\cal L}_2,{\cal L}_3,{\cal L}_4)^T$.

In matrix form, the linearized transformed system is now equivalent to:
\begin{equation}\label{sys.laplace.matrix}
\left[ \left( z+\frac{1}{\overline{\tau}} \right) I - H(z) C_{\Phi} \right] {\cal L} = {\cal L}(0)
\end{equation}
where $C_{\Phi}$ is obtained from the weight matrix, where each row $j$ was multiplied by $\phi_j$.

\begin{figure}
\centering
\includegraphics[width=0.75\textwidth]{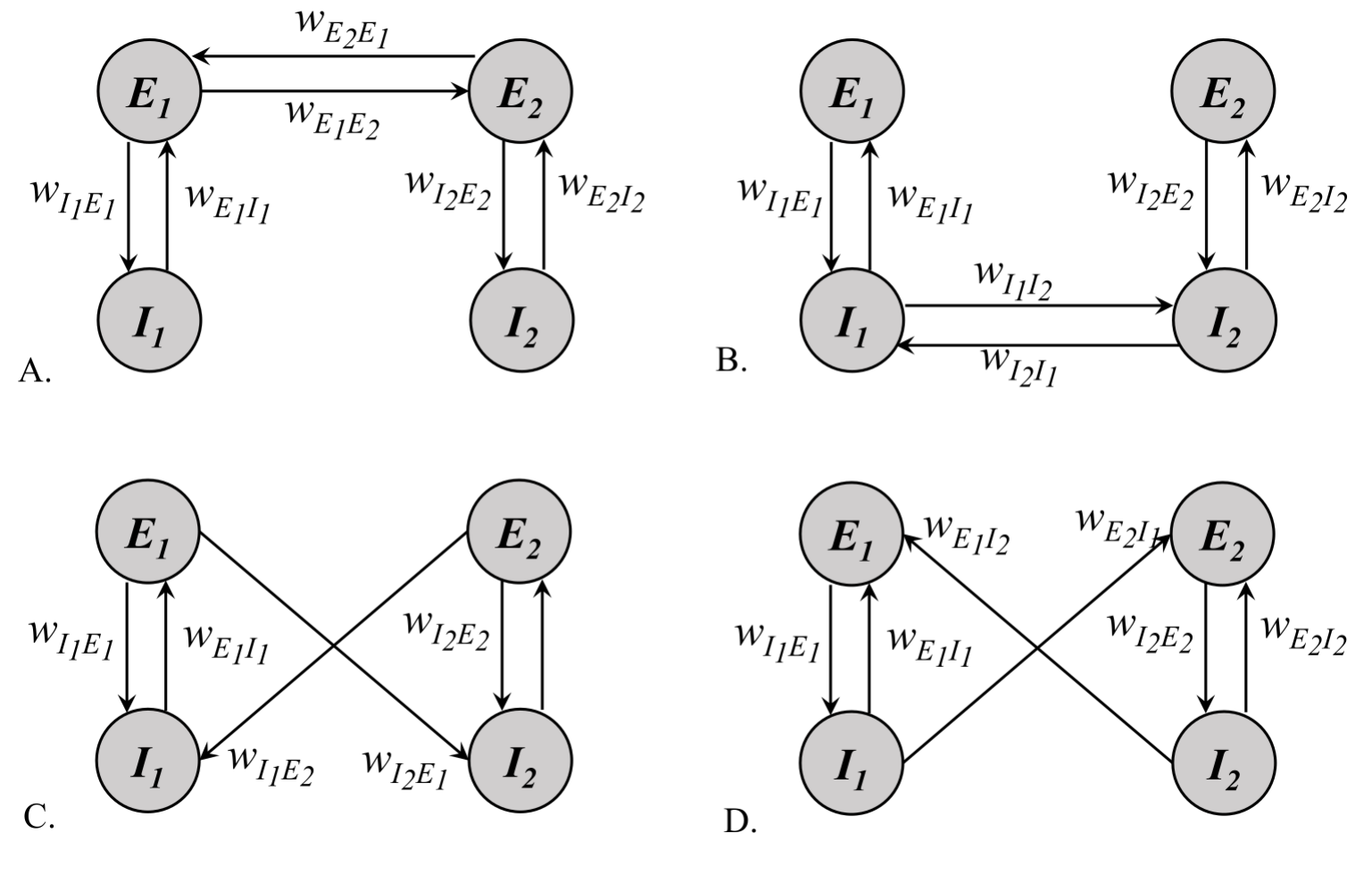}
\caption{\small \emph{{\bf Connectivity schemes between WC modules} studied in this project. {\bf A.} $E_1 \longleftrightarrow E_2$; {\bf B.} $I_1 \longleftrightarrow I_2$; {\bf C.} $E_1 \to I_2$  and $E_2 \to I_1$; {\bf D.} $I_1 \to E_2$  and $I_2 \to E_1$. The corresponding expressions for the coefficients $\alpha$ and $\beta$ are shown in Table~\ref{table_coefficients}.}}
\label{circuits}
\end{figure}

In this paper, we will focus in particular on investigating symmetric adjacency schemes in which each node in either the pair $(E_1,E_2)$ or in the pair $(I_1,I_2)$ is coupled with one node on the opposite side. This generates four coupling possibilities, as illustrated in Figure~\ref{circuits}: $E_1 \longleftrightarrow E_2$; $I_1 \longleftrightarrow I_2$; $E_1 \to I_2$ and $E_2 \to I_1$; $I_1 \to E_2$ and $I_2 \to E_1$. In all four cases, the characteristic equation of the system around the equilibrium ${\cal E}^*$ can be written in the form:
\begin{equation}
 \left(z+\dfrac{1}{\otau}\right)^4 - \alpha \left(z+\dfrac{1}{\otau}\right)^2\left( \dfrac{1}{\otau}\right)^2H^2(z) + \beta \left( \dfrac{1}{\otau} H(z)\right)^4=0 
\label{chareq_general}
\end{equation}
where the expressions of the coefficients $\alpha$ and $\beta$ with respect to the weights and the values of $\phi_j$ depend on the case. We will illustrate how $\alpha$ and $\beta$ are obtained in the first case; the other cases are similar, and we will only provide the final expressions in Table~\ref{table_coefficients}.

In the case $E_1 \longleftrightarrow E_2$, the weight matrix is of the form:
$$C = \left( 
\begin{array}{cccc}
0 & w_{_{E_1I_1}} & w_{_{E_1E_2}} & 0\\
w_{_{I_1E_1}} & 0 & 0 & 0\\
w_{_{E_2E_1}} & 0 & 0 & w_{_{E_2I_2}}\\
0 & 0 & w_{_{I_2E_2}} & 0\\
\end{array}
\right)$$
and the matrix of the linearized transformed system \eqref{sys.laplace.matrix} becomes:
$${\cal J} = \left( 
\begin{array}{cccc}
z + \frac{1}{\otau} & -\frac{1}{\otau} \phi_1 H w_{_{E_1I_1}} & -\frac{1}{\otau} \phi_1 H w_{_{E_1E_2}} & 0\\\\
-\frac{1}{\otau} \phi_2 H w_{_{I_1E_1}} & z + \frac{1}{\otau} & 0 & 0\\\\
-\frac{1}{\otau} \phi_3 H w_{_{E_2E_1}} & 0 & z + \frac{1}{\otau} & -\frac{1}{\otau} \phi_3 H w_{_{E_2I_2}}\\\\
0 & 0 & -\frac{1}{\otau} \phi_4 H w_{_{I_2E_2}} & z + \frac{1}{\otau}\\\\
\end{array}
\right)$$
Expanding the determinant, one easily gets the expressions for $\alpha$ and $\beta$ shown in Table~\ref{table_coefficients}. 

 \renewcommand{\arraystretch}{1.5}
\begin{table}[h!]
\begin{center}
\begin{tabular}{|l|l|}
\hline
Connectivity & Expressions for $\alpha$ and $\beta$\\
\hline 
  $E_1 \longleftrightarrow E_2$   &  $\alpha = \phi_3\phi_4w_{_{E_2I_2}}w_{_{I_2E_2}} + \phi_1\phi_3w_{_{E_1E_2}}w_{_{E_2E_1}} + \phi_1\phi_2w_{_{E_1I_1}}w_{_{I_1E_1}}$ \\
        &  $\beta = \phi_1\phi_2\phi_3\phi_4w_{_{E_1I_1}}w_{_{I_1E_1}}w_{_{E_2I_2}}w_{_{I_2E_2}}$\\
        \hline
$I_1 \longleftrightarrow I_2$   &  $\alpha = \phi_3\phi_4w_{_{E_2I_2}}w_{_{I_2E_2}} + \phi_2\phi_4w_{_{I_1I_2}}w_{_{I_2I_1}} + \phi_1\phi_2w_{_{E_1I_1}}w_{_{I_1E_1}}$ \\
        &  $\beta = \phi_1\phi_2\phi_3\phi_4w_{_{E_1I_1}}w_{_{I_1E_1}}w_{_{E_2I_2}}w_{_{I_2E_2}}$\\
        \hline
$E_1 \longrightarrow I_2$   &  $\alpha = \phi_3\phi_4w_{_{E_2I_2}}w_{_{I_2E_2}} + \phi_1\phi_2w_{_{E_1I_1}}w_{_{I_1E_1}}$ \\
 $E_2 \longrightarrow I_1$        &  $\beta = \phi_1\phi_2\phi_3\phi_4w_{_{E_2I_2}}w_{_{E_1I_1}}(w_{_{I_1E_1}}w_{_{I_2E_2}} - w_{_{I_1E_2}}w_{_{I_2E_1}} )$\\
        \hline
$I_1 \longrightarrow E_2$   &  $\alpha = \phi_3\phi_4w_{_{E_2I_2}}w_{_{I_2E_2}} + \phi_1\phi_2w_{_{E_1I_1}}w_{_{I_1E_1}}$ \\
 $I_2 \longrightarrow E_1$        &  $\beta = \phi_1\phi_2\phi_3\phi_4w_{_{I_2E_2}}w_{_{I_1E_1}}(w_{_{E_1I_1}}w_{_{E_2I_2}} - w_{_{E_2I_1}}w_{_{E_1I_2}} )$\\
        \hline
\end{tabular}
\end{center}
\caption{\small \emph{{\bf Dependence of the coefficients $\alpha$ and $\beta$ of the characteristic equation on the coupling strengths} for all four connectivity schemes considered in this paper.}}
\label{table_coefficients}
\end{table}

In general, it may be complicated to determine even the monotonicity of the dependence of $\alpha$ and $\beta$ on the connection weights, since the values of the $\phi_j$ depend themselves on these strengths. However, the signs of $\alpha$ and $\beta$ may be in some cases easy to establish, since ${\cal F}_j$ are increasing, hence $\phi_j$ are positive. This will later allow us to restrict our analysis to specific quadrants in the parameter space $(\alpha,\beta)$, depending on the connectivity scheme. Let's take a look, for example, at the third connectivity scheme. If the signs of the connections are the ones specified in the traditional model, then $w_{_{I_1E_1}},w_{_{I_2E_2}} > 0$, and $w_{_{E_1I_2}},w_{_{E_2I_1}},w_{_{E_1I_1}},w_{_{E_2I_2}} < 0$. This implies that $\alpha<0$ and the sign of $\beta$ is determined by comparing the level of the intra-modular excitation with the level of the inter-modular excitation, such that $\beta < 0$ iff  $w_{_{I_1E_1}}w_{_{I_2E_2}}<w_{_{I_2E_1}}w_{_{I_1E_2}}$.

\section{Stability and bifurcation analysis}
\label{stab_bif_analysis}

\subsection{The no-delay case}

Let's consider first the no-delay case (i.e. $H =1$, for all $z$). In this case, the characteristic equation \eqref{chareq_general} takes the form:
$$\left( z + \frac{1}{\otau} \right)^4 - \left( z + \frac{1}{\otau} \right)^2 \cdot \frac{\alpha}{\otau^2} + \frac{\beta}{\otau^4}=0 \quad \Leftrightarrow \quad (\otau z + 1)^4 - \alpha (\otau z+1)^2 + \beta =0.$$

\noindent Call $x = \otau z$, and notice that $\text{Re} (z) < 0$ when $\text{Re} (x) < 0$. Then the equation becomes, in expanded form:
$$x^4+4x^3+(6-\alpha)x^2 + (4-2\alpha)x +1-\alpha+\beta=0.$$

\noindent Applying the Routh-Hurwitz criterion, we get that the corresponding equilibrium is stable if and only if $6-\alpha >0$, $4-2\alpha>0$, $1-\alpha+\beta>0$, and $4(6-\alpha)(4-2\alpha)>(4-2\alpha)^2+16(1-\alpha+\beta)$. In conclusion, slightly rewriting, we have the following:

\begin{thm}
For the system without delays, the equilibrium is locally asymptotic stability  if and only if $(\alpha,\beta)$ belong to the parameter region 
$${\cal S} = \left\{ \alpha <2 \text{ and } \alpha -1< \beta < \frac{(\alpha-4)^2}{4} \right \}.$$
\end{thm}

In what follows, we primarily focus our analysis on the parameter region $\cal S$, with the aim of investigating whether asymptotic stability of the equilibrium can be lost by introducing a distributed time delay with a sufficiently large mean value.

\subsection{The case of a general delay distribution}

Returning to the system with distributed delays, we first want to establish whether there is a subset of the region ${\cal S}$ on which local asymptotic stability can still be guaranteed, regardless of the kernel $h(t)$. 

One can easily establish that the characteristic equation \eqref{chareq_general} can be expressed as:
\begin{equation}
F(z) := Q_{\ttau}^2(z)-\alpha Q_{\ttau}(z)+\beta=0
\label{char.eq.Q}
\end{equation}

\noindent where $\ds \ttau = \frac{\tau}{\overline{\tau}}$, and $\ds Q_{\ttau}(z) = \left( \frac{z+\ttau}{\ttau \hat{H}(z)} \right)^2$, with $\hat{H}(z) =H\left(\dfrac{z}{\tau}\right)$. In the following, we will assume that $\hat{H}(z)$ does not depend on $\tau$. For example, for the Dirac kernel, $\hat{H}(z)=e^{-z}$ and for the weak Gamma kernel, $\hat{H}(z)=(1+z)^{-1}$. 

Notice that, since $F(\infty) = \infty$ and $F(0) = 1-\alpha+\beta$, it follows that the equation \eqref{char.eq.Q} has a positive real root when $\beta<\alpha-1$, leading to an unstable equilibrium in this case. In fact, we have the following:

\begin{thm}[Delay-independent instability result]
    If $\beta<\alpha-1$, the equilibrium is unstable, regardless of the delay kernel $h(t)$.
\end{thm}

In fact, the line $\beta=\alpha-1$ is a saddle-node bifurcation curve, as seen in the following:

\begin{thm}[Saddle-node bifurcation result]
    A saddle-node bifurcation occurs for the equilibrium along the curve $\beta = \alpha-1$ if and only if $\alpha \neq 2$, regardless of the kernel $h(t)$. 
    \label{SNcurve}
\end{thm}

\begin{proof}
We already know that the characteristic equation has a zero root along the line $\beta = \alpha-1$. To insure that a saddle node bifurcation indeed occurs, we yet need to verify that the root $z=z(\beta)$ such that $z(\beta)=0$ when $\beta=\alpha-1$, satisfies the transversality condition 
$\ds \frac{dz}{d\beta} \neq 0$ as $\beta=\alpha-1$. We start with the characteristic equation and differentiate with respect to $\beta$ to get:
$$\frac{dz}{d\beta} \left[ 2Q_{\ttau}(z)Q'_{\ttau}(z) - \alpha Q'_{\ttau}(z) \right] + 1 = 0.$$
Allowing $z(\beta)=0$ as $\beta=\alpha-1$, and using the fact that $\hat{H}(0)=1$ and $\hat{H}'(0) = -1$, we get that $Q_{\ttau}(0)=1$ and $\ds Q'_{\ttau}(0) = 2 \left( 1 + \frac{1}{\ttau} \right)$. We use this to calculate
\begin{equation}
\left. \frac{dz}{d\beta} \right \rvert_{\beta=\alpha-1} = -\frac{1}{Q_{\ttau}'(0)[2Q_{\ttau}(0)-\alpha]} =  \frac{\ttau}{2(\ttau+1)(\alpha-2)}\neq 0,
\end{equation}
and hence, the transversality condition is verified.
\end{proof}

In particular, $\ds \left. \frac{dz}{d\beta} \right \rvert_{\beta=\alpha-1} < 0$ for $\alpha<2$. Hence crossing the line $\beta=\alpha -1$ upwards, one simple positive root of the characteristic equation becomes negative. In the latter part of this section, we will see that 
this is consistent with the exiting asymptotic stability region. With this in mind, a first simple result is presented below, which states sufficient conditions for stability of the equilibrium.

\begin{thm}[A first delay-independent stability result]
The equilibrium is locally asymptotically stable if $\lvert \alpha \rvert + \lvert \beta \rvert <1$, regardless of the kernel $h(t)$.
\label{thm.stability.simple}
\end{thm}

\begin{proof}
Suppose that the characteristic equation has a root $z$ such that $\text{Re}(z) \geq 0$. Since $\lvert\hat{H}(z) \rvert \leq 1$, it follows that 
$$\lvert Q_{\ttau}(z) \rvert \geq \frac{\lvert z+\ttau \rvert^2}{\ttau^2}= \frac{\lvert z \rvert^2+2\ttau \text{Re}(z)+\ttau^2}{\ttau^2} \geq 1.$$

But, from the characteristic equation \eqref{char.eq.Q} we have $\ds Q_{\ttau}(z) = \alpha - \frac{\beta}{Q_{\ttau}(z)}$, hence 
$$1 \leq \lvert Q_{\ttau}(z) \rvert \leq \lvert \alpha \rvert + \frac{\lvert \beta \rvert}{\lvert Q_{\ttau}(z) \rvert} \leq \lvert \alpha \rvert + \lvert \beta \rvert, $$
which is in contradiction with the inequality $|\alpha|+|\beta|<1$ from the statement of the Theorem. Hence, the proof is complete.
\end{proof}


It is clear from Theorem~\ref{SNcurve} and Theorem \ref{thm.stability.simple}  that, regardless of the form of $h(t)$, the system leaves the stability domain for the equilibrium when a saddle-node bifurcation line {\color{red}$\beta = 1-\alpha$} is crossed. We also know that this is the only saddle-node curve. We want to further characterise the delay-independent and delay-specific stability regions; we will investigate the boundaries of these domains to assess whether one can exit stability via a Hopf bifurcation, and for what values of the mean delay $\tau$.

For the system to cross through a Hopf bifurcation, the characteristic equation needs to have a pair of imaginary conjugate roots. In particular, there needs to exist an $\omega$ such that $z=i \omega$ is a root. In other words
\begin{equation}\label{eq.quadratic.Q}
Q_{\ttau}^2(i \omega) - \alpha Q_{\ttau}(i \omega) + \beta =0
\end{equation}

\noindent for some $\omega>0$. Recall that 
$$Q_{\ttau}(i \omega) =\left( \frac{i \omega + \ttau}{\ttau \hat{H}(i \omega)}\right)^2.$$

\noindent If we write the  Laplace transform in polar form as $\ds \hat{H}(i \omega) = \rho(\omega)e^{-i\theta(\omega)}$, then 
\begin{equation}Q_{\ttau}(i \omega) = \left(  1 + i\frac{\omega}{\ttau} \right)^2 \cdot \frac{1}{\rho^2(\omega)} \cdot e^{2i \theta(\omega)}=\left( 1 + \frac{\omega^2}{\ttau^2} \right) \cdot \frac{1}{\rho^2(\omega)} \cdot e^{\ds 2i \left[ \arctan\left( \frac{\omega}{\ttau}\right) + \theta(\omega) \right]}
\label{Q_polar}
\end{equation}

\noindent We distinguish two cases, based on the sign of the discriminant of the equation \eqref{eq.quadratic.Q} in $Q_{\ttau}$.\\

\noindent {\bf Case 1: $\ds \beta \geq \frac{\alpha^2}{4}$.} Then $Q_{\ttau}(i \omega)$ and $\overline{Q_{\ttau}(i\omega)}$ are complex conjugate roots of the polynomial $P(z) = z^2-\alpha z +\beta$, such that:
\begin{eqnarray}
    \alpha &=& 2 \text{Re} \left( Q_{\ttau}(i \omega)\right)
    \label{alpha}\\
    \beta &=& \lvert Q_{\ttau}(i \omega) \rvert^2
\end{eqnarray}

\noindent It then follows directly from~\eqref{Q_polar} that 
$$\sqrt{\beta} = \lvert Q_{\ttau}(i \omega) \rvert = \left( 1 + \frac{\omega^2}{\ttau^2} \right) \cdot \frac{1}{\rho^2(\omega)} \quad \Leftrightarrow \quad \frac{\omega^2}{\ttau^2} = \sqrt{\beta} \rho^2(\omega) -1.$$

\noindent Hence, a critical delay $\ttau^*(\omega)$ exists only if $\sqrt{\beta} \rho^2(\omega) -1 \geq 0$. Since $0<\rho(\omega) \leq1$ (based on basic properties of the Laplace transform), it follows that $\beta \geq 1$ is a necessary condition in this case for having a Hopf bifurcation. Consequently, the critical delay is given by:
\begin{equation}
    \ttau^*(\omega) = \frac{\omega}{\sqrt{\sqrt{\beta} \rho^2(\omega) -1}}\quad\text{with}\quad \rho(\omega)>\beta^{-1/4}
    \label{ttau_case1}
\end{equation}

\noindent and $\omega$ satisfying condition \eqref{alpha}, which can be rewritten as:
$$\alpha = 2\sqrt{\beta} \cos \left[ 2 \left( \arctan \frac{\omega}{\ttau} + \theta(\omega) \right)  \right]$$
or equivalently
\begin{equation}
    \cos \left[ 2\left( \arctan \sqrt{\sqrt{\beta} \rho^2(\omega)-1} +\theta(\omega)  \right) \right] = \frac{\alpha}{2\sqrt{\beta}}.
    \label{omega_case1}
\end{equation}

\vspace{2mm}
\noindent {\bf Case 2: $\ds \beta < \frac{\alpha^2}{4}$. } Then the quadratic polynomial $P(z)$ has real roots:
$$    Q_{\ttau}(i \omega) = \frac{\alpha \pm \sqrt{\alpha^2-4\beta}}{2}.$$
Hence,  from~\eqref{Q_polar} we deduce that there exists $k\in\mathbb{Z}^+$ such that
$$2 \left( \arctan \frac{\omega}{\ttau} + \theta(\omega) \right) = k\pi \quad \Leftrightarrow \quad \frac{\omega}{\ttau} = \tan \left( \frac{k \pi}{2} -\theta(\omega) \right) \quad \Leftrightarrow \quad \ttau(\omega) = \omega \cot \left( \frac{k\pi}{2} - \theta(\omega) \right).$$

It also follows directly from~\eqref{Q_polar} that 
$$Q_{\ttau}(i \omega) = \left(1 + \frac{\omega^2}{\ttau^2} \right) \cdot \frac{(-1)^k}{\rho^2(\omega)} = \frac{(-1)^k}{\ds \rho^2(\omega) \cos^2 \left( \frac{k\pi}{2} - \theta(\omega) \right)}  \quad \Rightarrow \quad \lvert Q_{\ttau}(i \omega) \rvert \geq 1.$$

Notice now that the real polynomial $P(x) = x^2-\alpha x +\beta$ has its turning point at $\ds x = \frac{\alpha}{2}$. Since we are in the stability region ${\cal S}$, we are under the condition that $\ds \frac{\alpha}{2}<1$. On the other hand, $P(1) = 1-\alpha+\beta >0$, hence $P$ has no real roots in $(1,\infty)$. Hence a real root $Q_{\ttau}(i \omega)$ with $\lvert Q_{\ttau}(i \omega) \rvert >1$ automatically needs to satisfy $Q_{\ttau}(i \omega) <-1$, making $k$ odd, and thus implying that
\begin{eqnarray}
    \ttau(\omega) &=& \omega \tan \theta(\omega) \quad\text{ with }\quad \tan \theta(\omega) >0 \quad\text{ and }
    \label{ttau_case2}\\
    Q_{\ttau}(i\omega) &=& -\frac{1}{\rho^2(\omega)\sin^2\theta(\omega)}=\frac{\alpha \pm \sqrt{\alpha^2-4\beta}}{2}.
\end{eqnarray}

Based on the sign of $P(-1) = 1+\alpha+\beta$ and on the value of $\alpha$, we distinguish three cases:

\begin{enumerate}[(i)]
\item If $P(-1)>0$ and $\alpha >-2$ (i.e. turning point occurs to the right of $x=-1$), then both roots $r_{1,2}$ of $P$ are in $(-1,1)$. Hence $Q_{\ttau}(i \omega)$ cannot be either of these roots. The system cannot undergo a Hopf bifurcation for any value of $\ttau$, hence the equilibrium remains locally asymptotically stable in the region $\alpha>-2$, $\beta>-\alpha-1$.

\item If $P(-1)>0$ and $\alpha<-2$, then both roots $r_{1,2}$ are in $(-\infty,-1]$, and $\omega$ can be such that $Q_{\ttau}(i \omega)=r_1$ or $Q_{\ttau}(i \omega)=r_2$.

\item If $P(-1)<0$, then $r_1$ is in $(-\infty,-1]$ and $r_2$ is in $(-1,1)$, and $Q_{\ttau}(i \omega)=r_1$ only.
\end{enumerate}

In both cases (ii) and (iii) the first positive root of the equation in $\omega$ (if such a root exists) will occur when $Q_{\ttau}(i \omega)=r_1$, or equivalently:
\begin{equation}
\lvert \rho(\omega)\sin \theta(\omega)  \rvert = \sqrt{-\frac{1}{r_1}} = \sqrt{\frac{2}{|\alpha-\sqrt{\alpha^2-4\beta}|}}\quad\text{with}\quad\tan \theta(\omega)>0.
\label{omega_case2}
\end{equation}

\noindent To insure that a Hopf bifurcation actually does occur with respect to $\ttau$ at the critical point of complex imaginary roots $\pm z(\ttau^*) = \pm i\omega$, we need to verify the transversality condition $\ds \text{Re} \left( \frac{dz}{d\ttau} \right) \neq 0$. For simplicity, we will abbreviate in the following $Q = Q_{\ttau}(z(\ttau))$. We know that $\ds Q^2-\alpha Q + \beta =0$, and $\ds Q = \left( \frac{z+\ttau}{\ttau \hat{H}(z)} \right)^2$.
We differentiate with respect to $\ttau$ to obtain:
$$2Q \left( \frac{\partial Q}{\partial \ttau} + \frac{\partial Q}{\partial z} \cdot \frac{dz}{d \ttau} \right) - \alpha \left( \frac{\partial Q}{\partial \ttau} + \frac{\partial Q}{\partial z} \cdot \frac{dz}{d \ttau} \right) =0$$
Hence
$$\frac{dz}{d \ttau} = -\frac{\ds \frac{\partial Q}{\partial \ttau}}{\ds \frac{\partial Q}{\partial z}} = \frac{z\hat{H}(z)}{\ttau(\hat{H}(z)-(z+\ttau)\hat{H}'(z))} = \frac{z}{\ds \ttau \left( 1-(z+\ttau)\frac{\hat{H}'(z)}{\hat{H}(z)} \right)}$$
Evaluating at the critical point:
$$\left. \frac{dz}{d \ttau} \right \rvert_{\ttau=\ttau^*}= \frac{i \omega}{\ds \ttau^* \left( 1-(i \omega+\ttau^*) \frac{\hat{H}'(i \omega)}{\hat{H}(i \omega)} \right)} $$

Since $\hat{H}(i \omega) = \rho(\omega) e^{-i\theta(\omega)}$, then $\hat{H}'(i \omega) = -[\rho(\omega) \theta'(\omega) + i\rho'(\omega)]e^{-i\theta(\omega)}$ and
$$\left. \frac{dz}{d \ttau} \right \rvert_{\ttau=\ttau^*}= \frac{i \omega}{\ds \ttau^* \left[ 1+(i \omega+\ttau^*)\left( \theta'(\omega) + i \frac{\rho'(\omega)}{\rho(\omega)} \right) \right]}$$

For simplicity of notation, let's call $\ds G(\omega) = 1+(i \omega+\ttau^*)\left( \theta'(\omega) + i \frac{\rho'(\omega)}{\rho(\omega)} \right)$. Then the real part
\begin{equation}
\text{Re}\left( \left. \frac{dz}{d \ttau} \right \rvert_{\ttau=\ttau^*} \right) = \frac{\omega}{\ttau^*} \cdot \frac{\text{Re}(i \overline{G(\omega)})}{\| G(\omega) \|^2} = - \frac{\omega}{\ttau^*} \cdot \frac{\text{Im}(\overline{G(\omega)})}{\| G(\omega) \|^2} = \frac{\omega}{\ttau^*} \cdot \frac{\ds \ttau^* \frac{\rho'(\omega)}{\rho(\omega)} +\omega \theta'(\omega)}{\| G(\omega) \|^2}
\label{transvers}
\end{equation}

The nonzero (transversality) condition is satisfied, marking a Hopf bifurcation at the critical $\ttau^*$, if the factor $\ds \ttau \frac{\rho'(\omega)}{\rho(\omega)} +\omega \theta'(\omega)$ does not change sign, which depends on the form of the kernel. We can distinguish more explicitly what this condition becomes in the two cases already outlined, for which explicit expressions were computed for the critical delay.\\

\noindent {\bf Case 1: $\ds \beta \geq \frac{\alpha^2}{4}$.} Using equation~\eqref{ttau_case1}  that gives us the critical delay in this case, we can calculate
\begin{eqnarray}
\text{Re}\left( \left. \frac{dz}{d \ttau} \right \rvert_{\ttau=\ttau*} \right) &=& \frac{\omega^2}{\ttau^* \| G(\omega) \|^2} \left[ \frac{\rho'(\omega)}{\rho(\omega) \sqrt{\sqrt{\beta}\rho^2(\omega)-1}} + \theta'(\omega) \right] \nonumber \\
&& \nonumber \\
&=& \frac{\omega^2}{\ttau^* \| G(\omega) \|^2} \left[ \arctan \sqrt{\sqrt{\beta}\rho^2(\omega)-1} + \theta(\omega) \right]'
\label{transers_case1}
\end{eqnarray}

\noindent {\bf Case 2: $\ds \beta < \frac{\alpha^2}{4}$.} Since in this case $\ttau^* = \omega \tan \theta(\omega)$, then
\begin{eqnarray}
\text{Re}\left( \left. \frac{dz}{d \ttau} \right \rvert_{\ttau=\ttau*} \right) &=& \frac{1}{\| G(\omega) \|^2} \cdot \frac{1}{\tan \theta(\omega)} \left( \omega \tan\theta(\omega) \cdot \frac{\rho'(\omega)}{\rho(\omega)} + \omega \theta'(\omega) \right) \nonumber \\
&& \nonumber \\
&=& \frac{\omega}{\| G(\omega) \|^2} \big[  \ln ( \lvert \rho(\omega) \sin \theta(\omega) \rvert ) \big]'
\label{transvers_case2}
\end{eqnarray}

While it is difficult to locate / order solutions for~\eqref{omega_case1} and~\eqref{omega_case2} to find whether a critical value of $\ttau^*$ corresponding to a stability-switch Hopf bifurcation does occur and where, we use this discussion above to describe conditions that guarantee local asymptotic stability regardless of the kernel $h(t)$, as shown in Figure~\ref{independent_stability}. We will then also compute, for illustration and comparison, Hopf curves for two particular examples of distributed delays: Dirac and weak Gamma kernels.

\begin{figure}[h!]
\begin{center}
\includegraphics[width=0.5\textwidth]{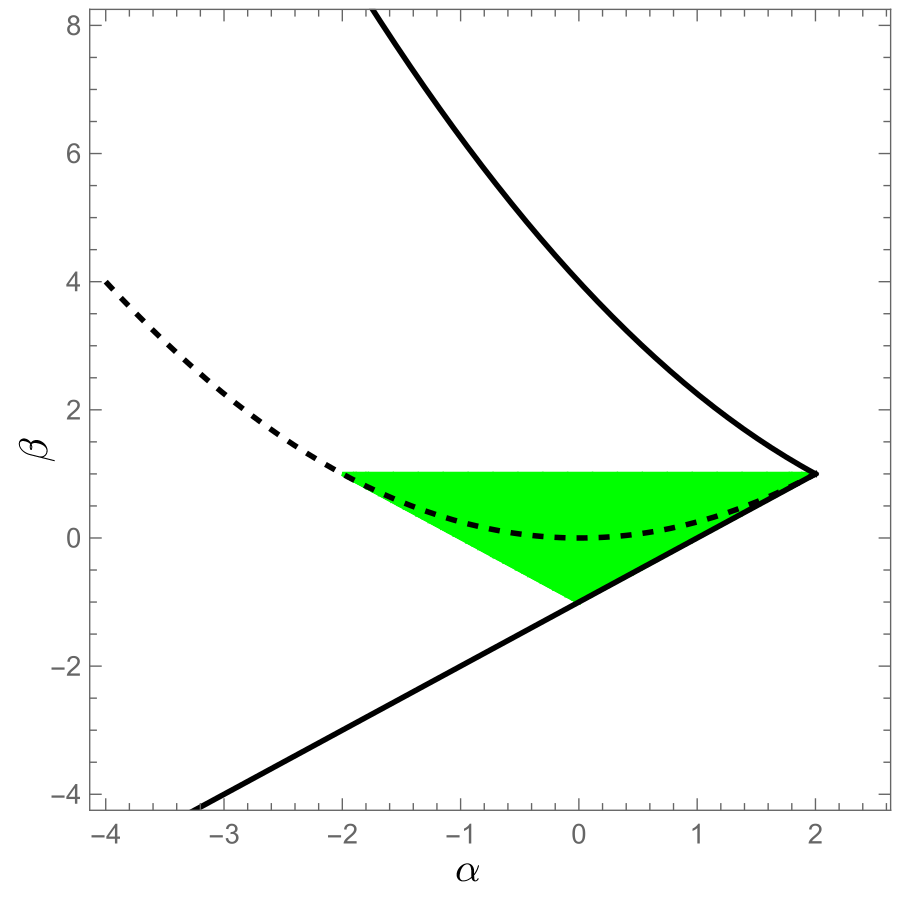}
\end{center}
\caption{\small \emph{{\bf Delay-independent stability region.} The region is shown in green. The curves $\ds \beta = \frac{(\alpha-4)^2}{4}$ and $\beta=\alpha-1$ are in solid black. The Parabola $\ds \beta = \frac{\alpha^2}{4}$ is the dashed black curve.}}
\label{independent_stability}
\end{figure}

\begin{thm}
    Local asymptotic stability of the equilibrium is guaranteed in the region 
    $${\cal R} = \{ \lvert \alpha \rvert -1 < \beta < 1 \} \subset {\cal S},$$
    regardless of the delay kernel $h(t)$.
    \label{stability_theorem}
\end{thm}

\begin{proof}
    The region ${\cal R}$ is easily obtained from the conditions for ${\cal S}$, together with the additional conditions $\beta<1$ and $\beta > -\alpha-1$ obtained above.
\end{proof}

\subsection{Illustration for the Dirac kernel}

For discrete delays, we have that $h(t) = \delta(t-\tau)$, $\hat{H}(z) = e^{-z}$, and hence $\rho(\omega)=1$ and $\theta(\omega)=\omega$. 

In the case $\ds \beta \geq \frac{\alpha^2}{4}$, equation~\eqref{omega_case1} becomes
$$\cos \left[ 2\left( \arctan \sqrt{\sqrt{\beta}-1} + \omega \right) \right] = \frac{\alpha}{2\sqrt{\beta}} \quad \Rightarrow \quad \omega_k^{\pm} = \pm \frac{1}{2} \arccos \frac{\alpha}{2\sqrt{\beta}} + k\pi - \arctan \sqrt{\sqrt{\beta}-1},~k\in\mathbb{Z}.$$
Therefore, in this case, we have an infinity of critical values for the time delay, provided by \eqref{ttau_case1}:
\[\ttau^*(\omega_k^\pm)=\frac{\omega_k^\pm}{\sqrt{\sqrt{\beta}  -1}}.\]
It is easy to verify that the positive roots $\omega_k^\pm$ are ordered as follows: 
\[0<\omega_0^+<\omega_1^-<\omega_1^+<\omega_2^-<\omega_2^+<\omega_3^-<\ldots\]
providing an increasing sequence of critical time delays. 

The smallest critical frequency is obtained for $k=0$, and the $(+)$ branch:
$$\omega_0^+ = \frac{1}{2} \arccos \frac{\alpha}{2\sqrt{\beta}} - \arctan \sqrt{\sqrt{\beta}-1},$$
hence, the first critical time delay can be computed from \eqref{ttau_case1} which gives:
$$
\ttau^*(\omega_0^+) = \frac{\omega_0^+}{\sqrt{\sqrt{\beta}  -1}}.
$$   

In the case $\ds \beta < \frac{\alpha^2}{4}$, equation~\eqref{omega_case2} becomes
$$\lvert \sin(\omega) \rvert= \sqrt{\frac{2}{\lvert \alpha - \sqrt{\alpha^2-4\beta} \rvert}} \quad \Rightarrow \quad \omega_k^{\pm} = \pm \arcsin \sqrt{\frac{2}{\lvert \alpha - \sqrt{\alpha^2-4\beta} \rvert}}+k\pi,~k\in\mathbb{Z}.$$
It is easy to see that we have the same ordering of the critical frequencies (and the corresponding critical time delays $\ttau^*(\omega_k^\pm)$) as in the previous case. Once again, the smallest critical frequency is 
\[\omega_0^+=\arcsin \sqrt{\frac{2}{\lvert \alpha - \sqrt{\alpha^2-4\beta} \rvert}}\]
and the first critical delay is deduced from~\eqref{ttau_case2}:
$$\ttau^*(\omega_0^+) = \omega_0^+ \tan \omega_0^+ = \arcsin \sqrt{\frac{2}{\lvert \alpha - \sqrt{\alpha^2-4\beta} \rvert}} \cdot \sqrt{\frac{2}{\lvert \alpha-\sqrt{\alpha^2-4\beta}\rvert - 2}}.$$

To show that Hopf bifurcations actually do occur at the corresponding critical delay values
$\ttau^*$, we verify the transversality condition~\eqref{transvers} for $\rho(\omega)=1$, and $\theta(\omega)=\omega$:
$$\text{Re}\left( \left. \frac{dz}{d \ttau} \right \rvert_{\ttau=\ttau^*} \right) 
 = \frac{\omega^2}{\ttau^* \| G(\omega) \|^2} > 0$$

 This tells us not only that a Hopf bifurcation does occur at every critical $\ttau$ where complex imaginary roots occur, but that the local asymptotic stability of the equilibrium is lost at $\ttau^*(\omega)$ corresponding to the smallest positive $\omega$ that satisfies equations~\eqref{omega_case1} or~\eqref{omega_case2}, depending on the case, as illustrated in Figure~\ref{first_crit_delay}a. Then the stability continues to degrade at the subsequent bifurcations, without the possibility of further stability switching.

\begin{figure}[h!]
\begin{center}
\includegraphics[width=\textwidth]{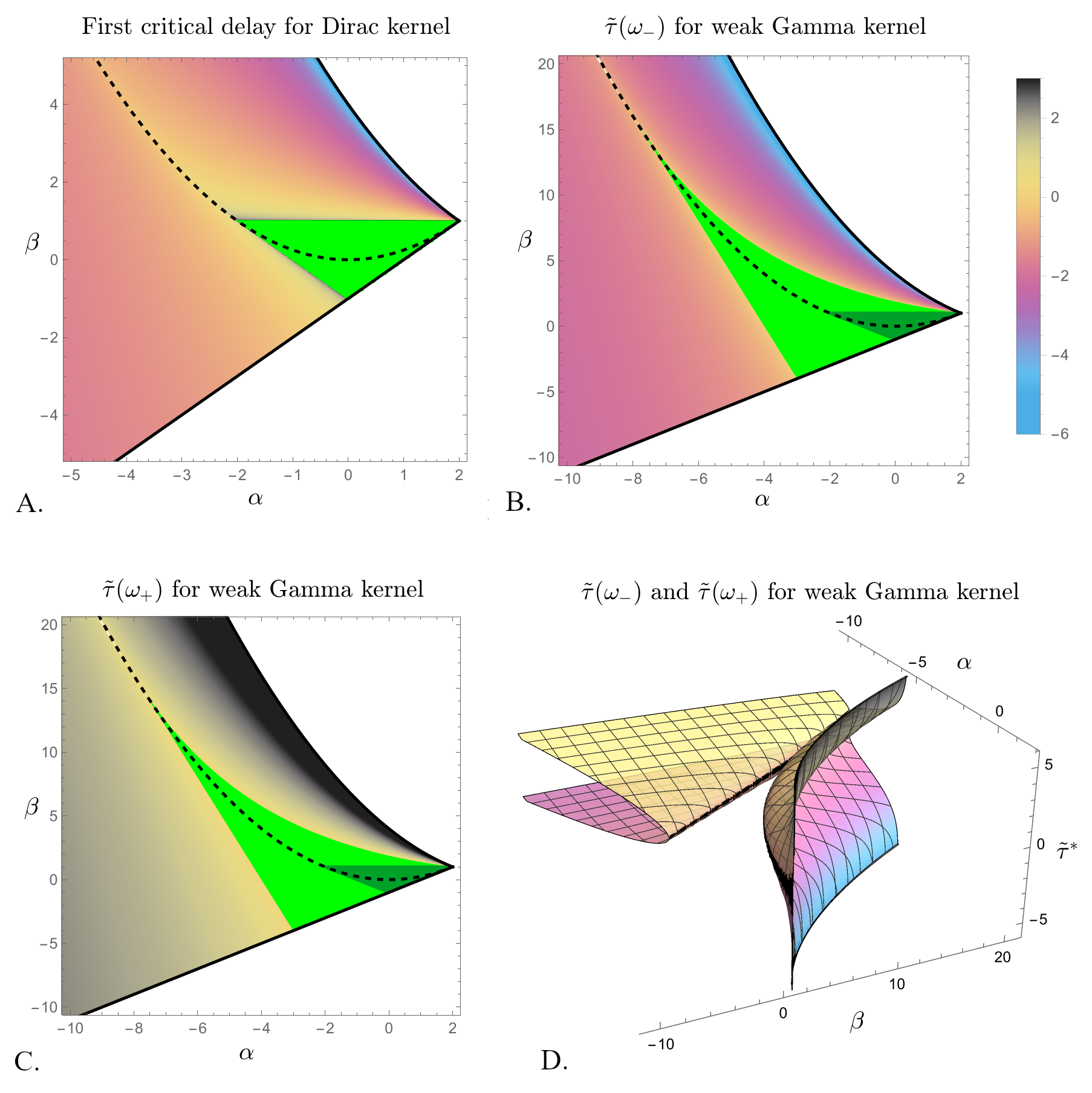}
\end{center}
\caption{\small \emph{{\bf Stability region and critical delays.} {\bf A.} Stability region for the Dirac kernel (shown in green) and first critical delay value, where onset of oscillations occurs at each point $(\alpha,\beta)$. For this kernel, oscillations are persistent (that is, further increasing $\ttau$ will not regaing stability of the equilibrium). The parabola $\ds \beta = \frac{(\alpha-4)^2}{4}$ (boundary of the domain ${\cal S}$) and the line $\beta = \alpha-1$ (saddle-node curve) are shown in solid black. The parabola $\beta = \ds \frac{\alpha^2}{4}$ (that distinguishes between the two computational cases) is shown as a dotted black curve. {\bf B.} Stability region for the weak Gamma kernel, and first critical delay value $\ttau(\omega_{-})$,where oscillations are triggered when increasing $\ttau$. The stability region for the Dirac kernel is also shown in dark green, for comparison. {\bf C.} 
Stability region for the weak Gamma kernel, and second critical delay value $\ttau(\omega_{-})$,where oscillations stop with increasing $\ttau$. {\bf D.} Critical delays for onset and stop of oscillations ($\ttau^*(\omega_{-})$ and $\ttau^*(\omega_{+})$) are shown simultaneously as surfaces with respect to $(\alpha,\beta)$. In all cases, the values of the critical delays are plotted on a logarithm scale, according to the color bar included.  }}
\label{first_crit_delay}
\end{figure}

 \subsection{Illustration for the weak Gamma kernel}

If $h(t)$ is the weak Gamma kernel, we have that $\ds \rho(\omega) = \frac{1}{\sqrt{\omega^2+1}}$, and $\theta(\omega) = \arctan \omega$. 

Reinterpreting Equations~\eqref{omega_case1} and~\eqref{omega_case2} in this case, one obtains that stability is guaranteed in the $(\alpha,\beta)$ region enclosed by (1) the line $\beta = -4\alpha-16$, (2) the parabola $\alpha = 2 \left( 8 - 8 \sqrt[4]{\beta} + \sqrt{\beta} \right)$ and (3) the line $\beta = \alpha -1$.

The computations leading to this result are somewhat technical, and are detailed in Appendix A. The region is shown in green in Figure~\ref{first_crit_delay}b,and it includes the kernel-invariant stability region in Thoerem~\ref{stability_theorem}. 


Outside the green region, the equilibrium will lose stability and the system will enter stable oscillations through a supercritical Hopf bifurcation when increasing the delay past a critical value $\ttau^*(\omega_{-})$. Conversely, the system undergoes a subcritical Hopf bifurcation, transition out of oscillations and regain stability of the equilibrium when the delay further increases past a larger critical value $\ttau^*(\omega_{+})$. The precise expressions for the $\omega_{\mp}$ corresponding to these critical delays are computed in the appendix, separately for parameters above and below the parabola $\ds \beta = \frac{\alpha^2}{4}$.

In the region outside the stability locus and below the parabola $\ds \beta = \frac{\alpha^2}{4}$ (cyan region in Figure~\ref{regions}) one can compute the transversality condition at a critical $\tau^*$, according to~\eqref{transvers_case2}, as
\begin{equation*}
\text{Re}\left( \left. \frac{dz}{d \ttau} \right \rvert_{\ttau=\ttau*} \right)
= \frac{\omega}{\| G(\omega) \|^2} \left(  \ln \frac{\omega^2}{1+\omega^2} \right)' = \frac{1-\omega^2}{\omega (1+\omega^2) \| G(\omega) \|^2} \neq 0
\end{equation*}
Since $\omega_{-}<1$ and $\omega_{+}>1$, as provided by~\eqref{cyan_region} in the Appendix, we have that $\ds \text{Re}\left( \left. \frac{dz}{d \ttau} \right \rvert_{\ttau=\ttau(\omega_{-})} \right)>0$ and  $\ds \text{Re}\left( \left. \frac{dz}{d \ttau} \right \rvert_{\ttau=\ttau(\omega_{+})} \right)<0$. Hence the equilibrium undergoes a supercritical Hopf bifurcation when increasing the delay past $\tau=\tau(\omega_{-})$ and a subcritical bifurcation at $\tau=\tau(\omega_{+})$. \\

In the region outside the stability locus and above the parabola $\ds \beta = \frac{\alpha^2}{4}$, the transversality condition for the weak Gamma kernel can be written as:
\begin{equation*}
\text{Re}\left( \left. \frac{dz}{d \ttau} \right \rvert_{\ttau=\ttau*} \right) = \frac{\omega^2}{\ttau^* \| G(\omega) \|^2} \left[ \arctan \sqrt{\frac{\sqrt{\beta}}{\omega^2+1}-1} + \arctan \omega \right]' = \frac{1}{\omega^2+1} \left[ 1- \frac{w}{\sqrt{\frac{\sqrt{\beta}}{\omega^2+1} -1}} \right] \neq 0
\end{equation*}

It is relatively easy to see that the expression only has one positive root $\omega = u_2 = \sqrt{-1+\sqrt[4]{\beta}}$. It is shown in the Appendix that $\omega_{-}<\sqrt{-1+\sqrt[4]{\beta}}<\omega_{+}$. Moreover, in the subregion where two additional roots $v_{\mp}$ are defined, we also have $\omega_{-}<v_{-}<\sqrt{-1+\sqrt[4]{\beta}}<v_{+}<\omega_{+}$. It follows that, in this case as well, the system can only undergo a pair of supercritical/subcritical Hopf bifurcations at $\ttau^*(\omega_{\mp})$, respectively, as the delay is increased, and that there can be no further stability switching at $\ttau^*(v_{\mp})$.


The dependence on $(\alpha,\beta)$ of the critical delays $\ttau^*(\omega_{\mp})$ is illustrated in Figure~\ref{first_crit_delay}b and c respectively. The three-dimensional plot in Figure~\ref{first_crit_delay}d provides a simultaneous representation of the $\ttau^*(\omega_{\mp})$, allowing for a better visualization of the mechanism by which delay changes can start and cease oscillations at various values of $\alpha$ and $\beta$.


Notice that, while in the case of the Dirac kernel the system cannot switch out of oscillations once these are triggered by further increasing the delay $\tau$, the situation is different for the weak Gamma kernel. This is an important qualitative difference from the Dirac kernel, in which increasing the delay can never lead an oscillating system to regain stability of the equilibrium. In contrast, increasing the delay seems to be able to act when needed as an efficient mechanism for ceasing oscillations in the case of the weak Gamma kernel. One can speculate that this distinction may be used a priori by neural population to decide which kind of delay to use when integrating inputs. The decision may be based on the functional importance allocated to being able to stop oscillations by simply increasing the average delay, without substantially altering the coupling weights.

\section{Applications to brain circuits}

In this section, we illustrate how our coupled Wilson-Cowan model with distributed delays can be used to understand dynamic behavior in brain circuits. Compared to other existing models of the same circuits, our approach adds a crucial timing/memory component to integrating inputs, as encompassed by the distributed delays. Hence our illustration and discussion will focus around this aspect, and how it interacts with other factors that contribute to the firing rate dynamics (such as coupling weights and signs, or connectivity geometry).

As previously mentioned, realistic brain networks are much larger, and feature complex structure and interactions that far transcend the model setup. However, in neural dynamic modeling it is often useful to consider dramatic simplifications of these networks, in order to study local, targeted questions, and to generate hypotheses that can be further tested in more realistic setups. Below, we consider two examples that have been studied before in the literature, in which the basic functional structure of the brain circuit is of one of the types considered in Figure~\ref{circuits}. Both examples aim to understand cortical feedback control in regulatory neural circuits, and its significance to healthy and pathological function. The first model considers interactions between the Subtalamic Nucleus (STN) and Globus Pallidus pars Externa (GPe), as well as their cortical regulation (as described in~\cite{wang2023possible}), and discusses the contribution to these mechanisms to generating the Beta-oscillations observed in Parkinson's Disease. The second model considers the feedback loop comprised of the prefrontal cortex and the basolatral amygdala, and the potential mechanisms behind the gamma oscillations that are believed to be critial in emotion regulation~\cite{headley2021gamma}.

\subsection{Modeling feedback projections from cortex on basal ganglia}

In~\cite{wang2023possible}, the authors use a cortex-basal ganglia resonance network to study the mechanism generating the Beta oscillations found in abnormal function, such as that observed in Parkinson's Disease. The reduced network (including projections between STN and GPe, together with regulatory connections from two cortical regions and feedback projections) takes the form in Figure~\ref{circuits}a. The circuit is schematically represented in Figure~\ref{Wang_circuit} with the specific names for variables, parameters and inputs used in the reference. The input integration sigmoidal function is specified as
$$\ds {\cal F}_j(u) = \frac{M_j}{\ds 1 + \left( \frac{M_j}{B_j} - 1 \right) \exp(-4u/M_j)}$$
A complete translation of nomenclature and parameter values between the reference and those used in our paper is included in Table~\ref{Table_Wang}.

The original model incorporates discrete delays, and identifies, for fixed values of the coupling parameters, a critical delay value $T^*$; if the delay is increased past $T=T^*$, a supercritical Hopf bifurcation occurs, with loss of stability for the equilibrium, and onset of stable oscillations. The authors then use this information to study in detail the effects of perturbing the inhibitory coupling strengths $W_{SC}$ (cortical projection to STN) and $W_{GS}$ (feedback projection from GPi to the cortex) away from their baseline values, when the system is operating in the proximity of the critical delay $T^*$ computed for the baseline values.

In our work, we will use the theoretical results in Section~\ref{stab_bif_analysis} to extend the analysis of the model to a more comprehensive discussion around the effects of distributed delays. After rephrasing some of the original results in light of our theoretical setup for the Dirac kernel, we show how these results differ when the delays are distributed according to a weak Gamma kernel. In addition, while in our theoretical analysis it was convenient to classify dynamic behaviors by navigating the $(\alpha,\beta)$ plane, for biological interpretations it is optimal to rephrase our results directly in terms of coupling weights.

\begin{figure}[h!]
\begin{center}
\includegraphics[width=0.7\textwidth]{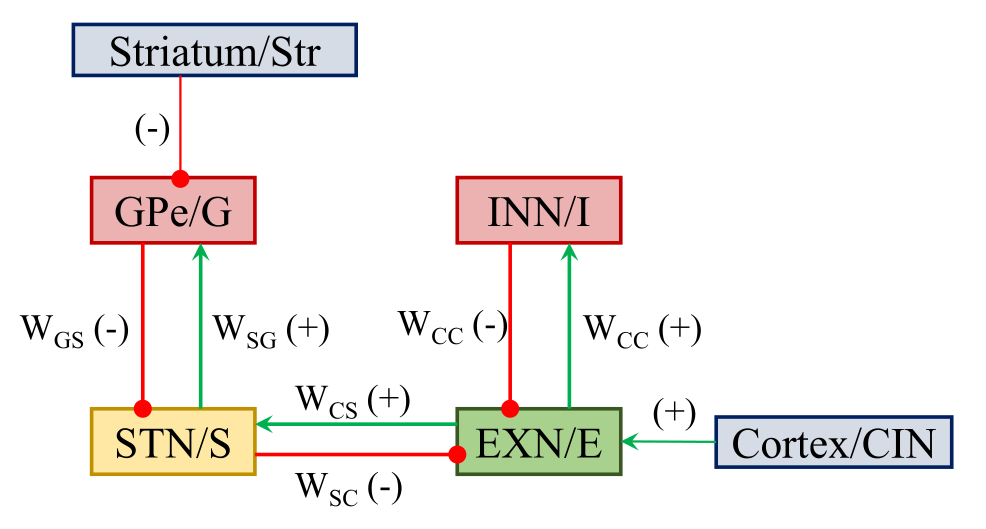}
\end{center}
\caption{\small \emph{{\bf Cortex-basal ganglia circuit,} as described in the reference~\cite{wang2023possible}. The system variables are s follows: $G(t)$ for the Globus Pallidus pars Externa (GPe); $S(t)$ for the Subtalamic Nucelus (STN); $E(t)$ for the excitatry cortical nucleus (EXN); $I(t)$ for the inhibitory cortical nucleus (INN). Constant external inputs are received from the Striatum (Str) and from other areas in the Cortex (CIN). The excitatory/inhibitory projections are designated with green and respectively red arrows. The connection weights and their sign are specified, where relevant. The values used for the connectivity strengths are provided in Table~\ref{Table_Wang}.}}
\label{Wang_circuit}
\end{figure}

\begin{table}[h!]
\centering
\begin{tabular}{|c|c|c|c|}
       \hline
\multicolumn{4}{|l|}{\bf Variables, functions and connectivities}\\  
\hline
Name & Relation reference & Value & Units \\
\hline
$E_{1}$,$I_1$,$E_2$,$I_2$  & $S,G,E,I$ & &\\
${\cal F}_j$  & $AF_j$ & &\\
$w_{_{E_{1}I_{1}}}$  & -$W_{GS}$ & -4.87 &\\
$w_{_{I_{1}E_{1}}}$  & $W_{SG}$ & 2.56 &\\
$w_{_{E_{1}E_{2}}}$  & $W_{CS}$ & 6.60 &\\
$w_{_{E_{2}E_{1}}}$  & -$W_{SC}$ & -2.58 & \\
$w_{_{E_{2}I_{2}}}$  & -$W_{CC}$ & -1.56 & \\
$w_{_{I_{2}E_{2}}}$  & $W_{CC}$ & 1.56 & \\
\hline
\hline
\multicolumn{4}{|l|}{\bf External inputs}\\
\hline
Name & Relation reference & Value & Units\\
\hline
$P_2$ & Str & 40.51 & spk/s\\
$P_3$ & CIN & 172.18 & spk/s\\
\hline
\end{tabular}
\begin{tabular}{|c|c|c|c|}
\hline
\multicolumn{4}{|l|}{\bf Sigmoidal parameters}\\    
\hline
Name & Relation reference & Value & Units \\
\hline
$M_1$  & $M_S$ & 300 & spk/s \\
$M_2$  & $M_G$ & 400 & spk/s \\
$M_3$  & $M_E$ & 71.77 & spk/s \\
$M_4$  & $M_I$ & 277.39 & spk/s \\
$B_1$  & $B_S$ & 17 & spk/s \\
$B_2$  & $B_G$ & 75 & spk/s \\
$B_3$  & $B_E$ & 3.62 & spk/s \\
$B_4$  & $B_I$ & 9.87 & spk/s \\
\hline
\hline
\multicolumn{4}{|l|}{\bf Time constants and delays}\\
\hline
Name & Relation reference & Value & Units\\
\hline
$\otau$ & $\tau_{_{G,S,E,I}}$ & 15 & ms\\
$\ttau$ & $T/\otau$ & &\\
\hline
\end{tabular}
\caption{\small \emph{{\bf Translation of variables, functions and parameters} from the reference model~\cite{wang2023possible} to the conventions used in our general analysis.}}
\label{Table_Wang}
\end{table}

We first use our framework to contextualize the effects of changing the strengths of the inhibitory coupling from the GPe to the STN, and of the feedback inhibition from the STN to the EXN, as described in the reference. For example, in the original analysis, $W_{SC}$ was varied when the value of $W_{GS}$ was first fixed to $W_{GS}=4.87$ and then to $W_{GS}=1.33$ (while all other parameters were fixed to the table values). The discrete delay was taken to be $T = 3.9492$ ms (which is the critical delay computed for the baseline values $W_{SC}=2.58$ and $W_{GS}=4.7$). In Figure~\ref{increase_WSC}a, we illustrate the two corresponding paths of increasing the magnitude of $W_{SC}$ (for the two fixed values of $W_{GS}$, respectively) in our $(\alpha,\beta)$ plane. To match and interpret the original simulations, we also show in the same figure panel the stability region for the Dirac kernel corresponding to $T = 3.9492$ (i.e., $\ttau=T/\otau=0.0026$). The delay-independent stability region for the Dirac kernel is also illustrated in bright green, for comparison. 

For the GPe projection weight fixed to $W_{GS}=1.33$, the parameter curve $(\alpha(W_{SC}),\beta(W_{SC}))$ starts at $W_{SC}=0$ within the stability region for the corresponding delay $T = 3.9492$ (light green region). As the feedback STN inhibition to the EXN is introduced and $W_{SC}$ increases from zero, the equilibrium quickly loses stability through a supercritical Hopf bifurcation, and the system transitions  into oscillations (blue region). These oscillations persist as the feedback inhibition is increased to $W_{SC} \sim 15$, where the curve crosses back through a supercritical Hopf bifurcation into the region of stability for the equilibrium, efficiently ceasing  oscillations. This bidirectional Hopf window is represented and discussed in Figure 5 of the reference. In the same figure, it is shown that the situation differs when fixing $W_{GS}=4.87$ and increasing $W_{SC}$, with oscillations that set in at $W_{SC}\sim 2.5$, and persist to the end of the interval. This suggests the possibility that, when the feedback modulation from the STN to the EXN is extremely low, increasing it acts as a trigger for oscillations. Furthermore, continuing to increase this feedback may act as a mechanism to stop oscillations, as long as the STN receives a strong enough inhibitory projection from the GPe. Illustrating this mechanism in the $(\alpha,\beta)$ plane allows us to put the result in perspective of our theoretical analysis, and further generalize it below.

\begin{figure}[h!]
\begin{center}
\includegraphics[width=\textwidth]{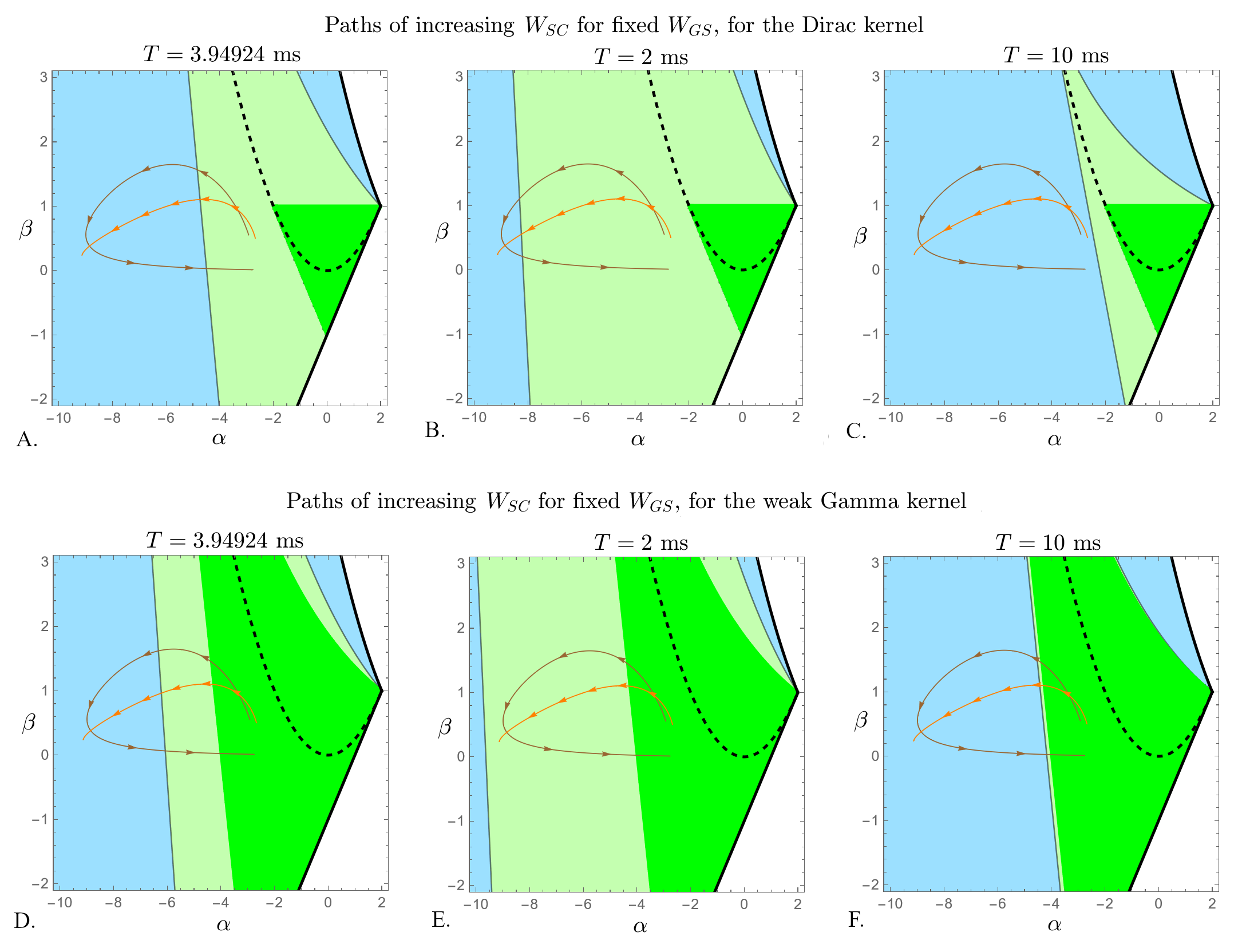}
\end{center}
\caption{\small \emph{{\bf Evolution of the stability of the equilibrium as $W_{SC}$ is increased from 0 to 20, for fixed $W_{GS}$.} The paths of increasing $W_{SC}$ are represented in the plane $(\alpha,\beta)$ as a brown curve (for $W_{GS}=1.33$) and as an orange curve (for $W_{GS}=4.87$). The top panels show these curves overlaid on the stability regions for the Dirac kernel, for three values of the discrete delay: {\bf A.} $T=3.94924$ ms (the value in the reference); {\bf B.} $T=2$ ms; {\bf C.} $T=10$ ms. The bottom panels show the same curves overlaid on the stability regions for the weak Gamma kernel, for the same three values of the distributed delay: {\bf D.} $T=3.94924$ ms (the value in the reference); {\bf E.} $T=2$ ms; {\bf F.} $T=10$ ms. In each case, the delay-independent stability region for the equilibrium is shown in bright green; the larger, delay-specific stability region is shown in light green. All other system parameters were fixed to their baseline values in Table~\ref{Table_Wang}.}}
\label{increase_WSC}
\end{figure}

Indeed, a natural follow-up question to consider is whether this mechanism for onset/stop oscillations is preserved for other values of the discrete delay $T$. To track how the results depend on the delay, we illustrate the corresponding transitions for a shorter delay $T=2$ (Figure~\ref{increase_WSC}b), and for a longer delay $T=10$ (Figure~\ref{increase_WSC}). These produce very different scenarios, with the shorter delay causing the system to remain mostly in the stability region along both paths of increasing $w_{SC}$, while in the case of longer delay, both paths are confound to the blue, oscillatory region. These observations suggest that if the cells in the circuit operate with discrete delays when integrating inputs, longer delays will promote oscillations.

One needs to recall, however, that the use of the Dirac kernel is only a convenient modeling choice, and that in reality neural populations may use more sophisticated, distributed delays when integrating inputs. In fact, when the Wilson-Cowan model was first introduced, the authors had suggested a weak Gamma kernel as an appropriate choice for the delay distribution. For our next step, we use the theoretical results in Section~\ref{stab_bif_analysis} to show how the same changes to the coupling strengths described previously for the Dirac kernel produce different dynamic effects in the case of the weak Gamma kernel. The bottom panels (d,e and f) of Figure~\ref{increase_WSC} illustrate the same paths as the top three panels, for the same values of the average delay $T$, except that this delay is now distributed according to the weak Gamma kernel. The most notable difference is that, compared to the Dirac kernel, the weak Gamma distribution tends to confine the dynamics to the stability regime. In fact, for short enough distributed delays, the system never enters oscillations (in panel e, corresponding to $T=2$, the two paths never leave the weak Gamma stability region). Increasing the distributed delay up to $T=10$ (panel f) leads to recovering a weak Gamma stability region comparable to that for $T=3.94924$ in the case of the Dirac kernel (panel a). One may further speculate that a failure of the neural populations to use an adequate delay distribution when processing inputs (e.g., using a Dirac instead of weak Gamma distribution) may be a contributing culprit to generating excessive oscillations.


\begin{figure}[h!]
\begin{center}
\includegraphics[width=\textwidth]{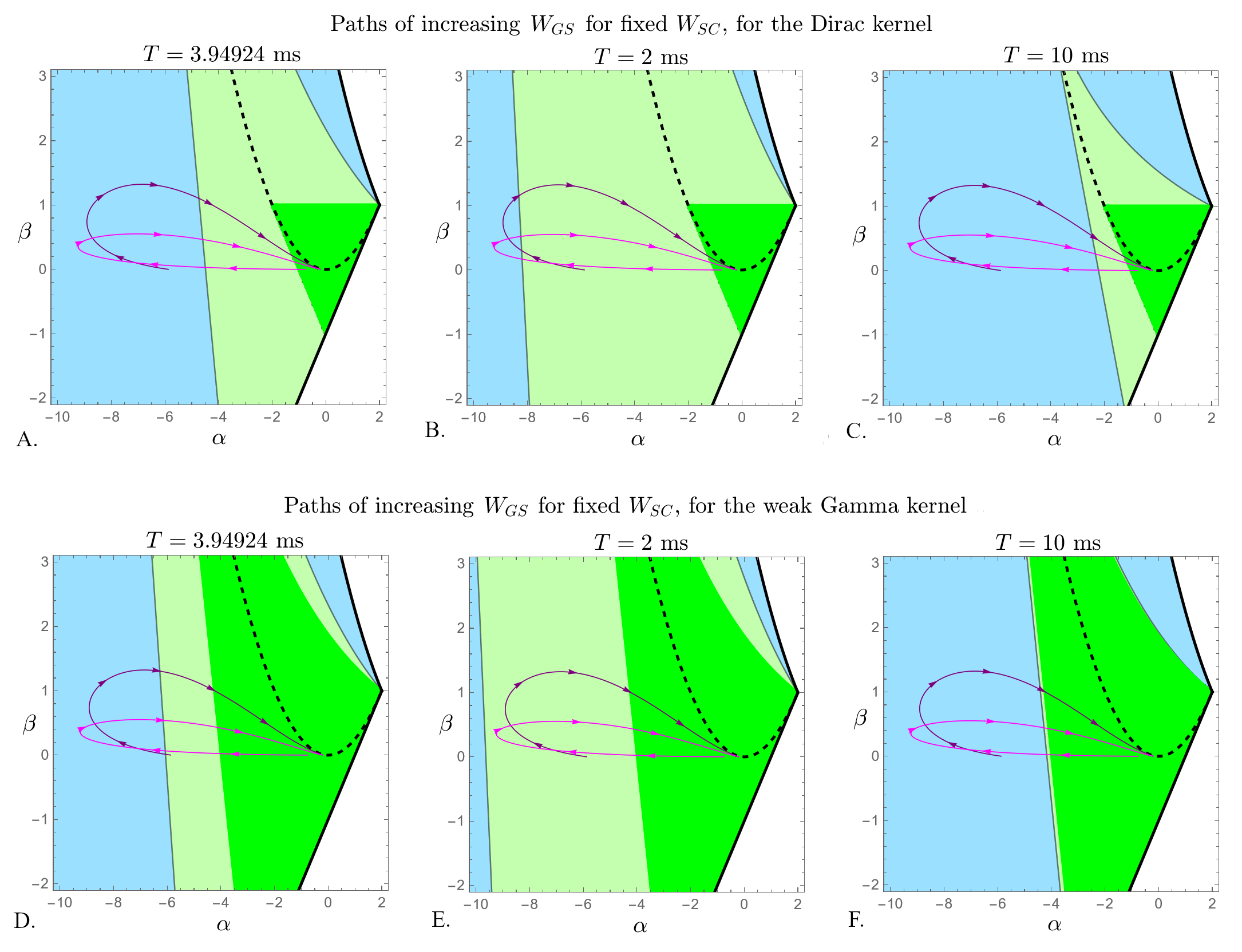}
\end{center}
\caption{\small \emph{{\bf Evolution of the stability of the equilibrium as $W_{GS}$ is increased, for fixed $W_{SC}$.} The paths of increasing $W_{GS}$ are represented in the plane $(\alpha,\beta)$ as a purple curve (for $W_{SC}=2.58$) and as a pink curve (for $W_{SC}=10$). The top panels show these curves overlaid on the stability regions for the Dirac kernel, for three values of the discrete delay: {\bf A.} $T=3.94924$ ms (the value in the reference); {\bf B.} $T=2$ ms; {\bf C.} $T=10$ ms. The bottom panels show the same curves overlaid on the stability regions for the weak Gamma kernel, for the same three values of the distributed delay: {\bf D.} $T=3.94924$ ms (the value in the reference); {\bf E.} $T=2$ ms; {\bf F.} $T=10$ ms. In each case, the delay-independent stability region for the equilibrium is shown in bright green; the larger, delay-specific stability region is shown in light green. All other system parameters were fixed to their baseline values in Table~\ref{Table_Wang}.}}
\label{increase_WGS}
\end{figure}

Another parameter experiment considered in the original reference is that of increasing $W_{GS}$, while keeping $W_{SC}$ fixed. Figure~\ref{increase_WGS} starts by replicating the behavior reported by the Figure 7 of the reference for the Dirac kernel. In our figure panel, we trace in the $(\alpha,\beta)$ plane the curves corresponding to fixing $W_{SC}=2.58$ (in purple) and $W_{SC}=10$ (in pink) as $W_{GS}$ is being increased. We overlay these curves on the level curve for the delay value used in the reference ($T=3.94924$). We observe how, in the case of lower $W_{SC}$, the systems transitions from a stable equilibrium (light green region) to oscillations (blue region); for the higher values of $W_{SC}$, the system also transitions back into the stability region if $W_{GS}$ is increased enough. In both cases, oscillations are possible only within a specific interval for $W_{GS}$, and perturbations outside of these limits would promptly cease the cycling. In panels b and c of the same figure, we show how this picture changes when the discrete delay is changed to $T=2$, then to $T=10$. As one may expect based on Figure~\ref{first_crit_delay}, the curves spend more time in the larger stability region, reducing the interval for $W_{GS}$ where  oscillations are promoted. Conversely, the system with a longer discrete delay will spend the whole time in the blue region, suggesting that a slower input integration renders the system more prone to oscillating. A similar association of a larger oscillatory interval for $W_{GS}$ when $T$ increases is associated with the weak Gamma kernel. However, just as before, weak Gamma distributed delays are altogether more efficient at keeping the system within the stability region than discrete delays.  

From both experiments illustrated in Figures~\ref{increase_WSC} and ~\ref{increase_WGS}, one may further speculate that mechanisms that lead to abnormal oscillations in the system may go beyond simple perturbations in the coupling parameters. Operating with an inadequate kernel, of with abnormally long delays in processing inputs may both represent root problems for the excessive oscillations in the GPe/STN system, as observed in pathologies like Parkinson's Disease. To better illustrate the direct interplay between all these factors in determining the ultimate behavior of the system, we show in Figure~\ref{WGS_WSC_plane} the map of the critical delays in the parameter plane $(W_{GS},W_{SC})$, for both Dirac and weak Gamma distributed delays. This delivers, for this particular application, a useful translation of the theoretical representation in terms of $(\alpha,\beta)$ from Figure~\ref{first_crit_delay}.

One can now more easily navigate the Dirac $(W_{GS},W_{SC})$ parameter plane in Figure~\ref{WGS_WSC_plane}a. It is noticeable that the level sets of $T^*$ are so that, when fixing $W_{SC}$ and increasing $W_{GS}$ far enough from zero, the system will almost always cross the $T^*$ threshold twice (once with onset and once with cessation of oscillations). On the other hand, fixing $W_{GS}$ and increasing $W_{SC}$ only appears to have a stop-oscillation built in mechanism for small values of $W_{GS}$ (such as $W_{GS}=1.33$, illustrated in Figure~\ref{increase_WSC}). For larger values of $W_{GS}$ (e.g., $W_{GS}=3.87$), the oscillations will persist no matter how high the STN-EXN inhibition. 

This remains true for the weak Gamma case, illustrated in Figure~\ref{WGS_WSC_plane}b. However, in the case of weak Gamma distributed delays, the critical delay profile is significantly higher than that obtained for the Dirac kernel (with a larger stability region, and colors corresponding to larger $T^*$ in the region where oscillations are possible). More precisely, for any given parameter pair $(W_{GS},W_{SC})$, the $T^*$ value where oscillations are triggered for the weak Gamma kernel is always higher than the Dirac critical value, implying that a discrete delay scheme will promote oscillations over weak Gamma distributed delays, for the same average delay value. This supports the observations for specific paths of increasing $W_{SC}$ and $W_{GS}$ from Figures~\ref{increase_WGS} and~\ref{increase_WSC}. 

\begin{figure}[h!]
\begin{center}
\includegraphics[width=0.95\textwidth]{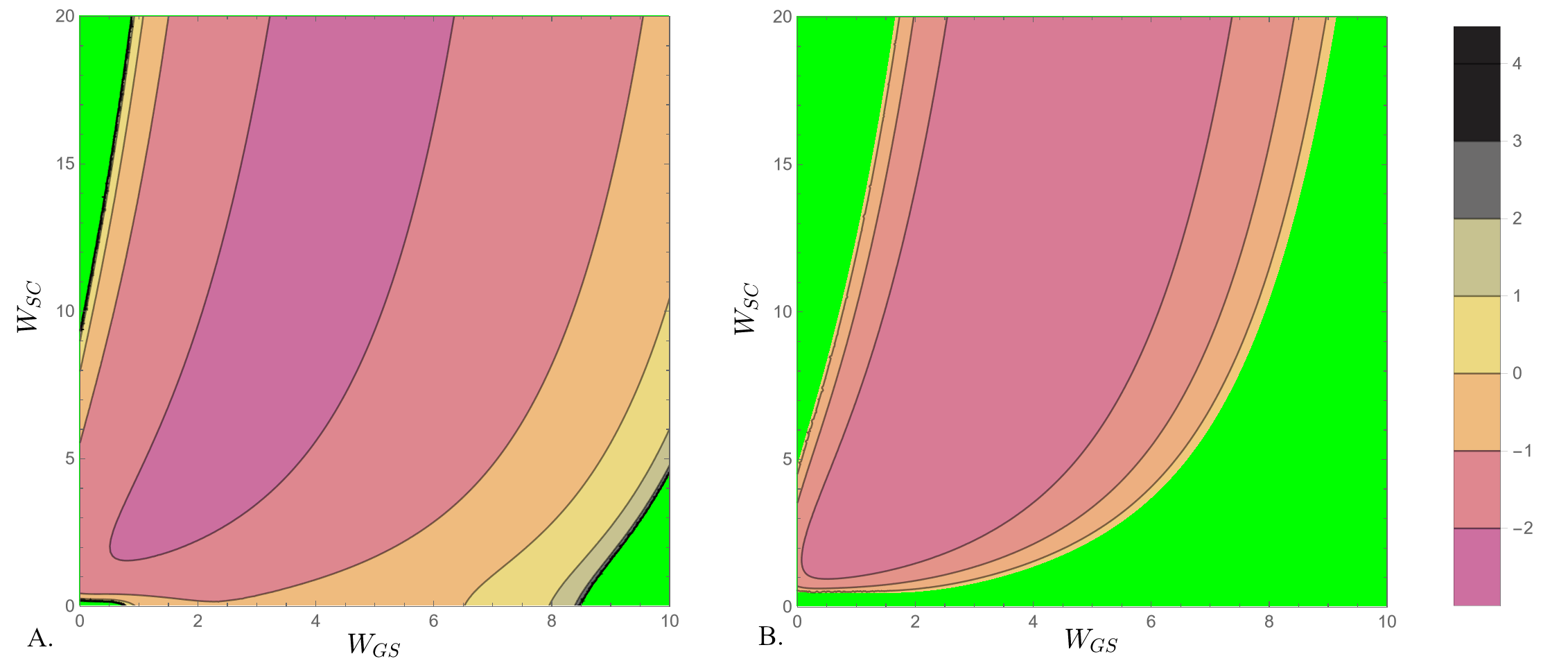}
\end{center}
\caption{\small \emph{{\bf Critical delay profiles in the $(W_{GS},W_{SC})$ parameter plane.} The panels show, for the Dirac ({\bf A})  and the weak Gamma kernel ({\bf B}) the logarithm of the critical delay $T^*$ at which transition into oscillations occurs when increasing $T$. The color coding, specified in the bar, is the same as in Figure~\ref{first_crit_delay}. As before, the logarithm was used to increase visibility. As in Figure~\ref{first_crit_delay}, the green region represents the stability region for the corresponding kernel. Outside of this region, a representation based on contour-curve partitioning was preferred in this case, to facilitate the comparison between kernels.}}
\label{WGS_WSC_plane}
\end{figure}

To fix our ideas, we further illustrate the comparison between kernels with simulations for specific parameter values and increasing delay $T$. Figure~\ref{simulation_Dirac} shows the GPe-STN projection of the solution with initial condition $(E_1,I_1,E_2,I_2) = (0,0,0,0)$ under three different scenarios dictated by the average delay $T$: before, at and after the supercritical Hopf bifurcation at $T^* = 3.94924$. This illustrates the formation of a stable cycle and the subsequent oscillations in basal ganglia firing rates. Figure~\ref{simulation_weak} shows the corresponding trajectories in the same two-dimensional GPe-STN slice, in the case of weak Gamma distributed delays, as $T$ is increased to pushed the system through a supercritical bifurcation at $T^*_{-}$ 
and then through a subcritical Hopf bifurcation at $T^*_{+}$. 
That is, as $T$ is increased, the system will first enter and then leave the oscillatory regime. It is important to recall that this oscillation \emph{window} did not appear in the Dirac kernel, where we found oscillations to persist under any delay increase once they are triggered. In contrast, for the weak Gamma kernel, increasing the delay $T$ past the upper threshold $T^*_{+}$ acts as a stop-oscillation mechanism. We note that this mechanism is relevant only if $T^*_{+}$ is within plausible biological values. As Figure~\ref{first_crit_delay}c shows, the position and width of the oscillation window depend on the coupling parameters; more precisely, $T^*_{-}$ increases and $T^*_{+}$ decreases close to the boundary of the stability region. Operating with coupling strengths that place the system close to this boundary may bring the stop-oscillation $T^*_{+}$ within a biological range, and would concomitantly increase the oscillation-triggering delay $T^*_{-}$. 

\begin{figure}[h!]
\begin{center}
\includegraphics[width=0.8\textwidth]{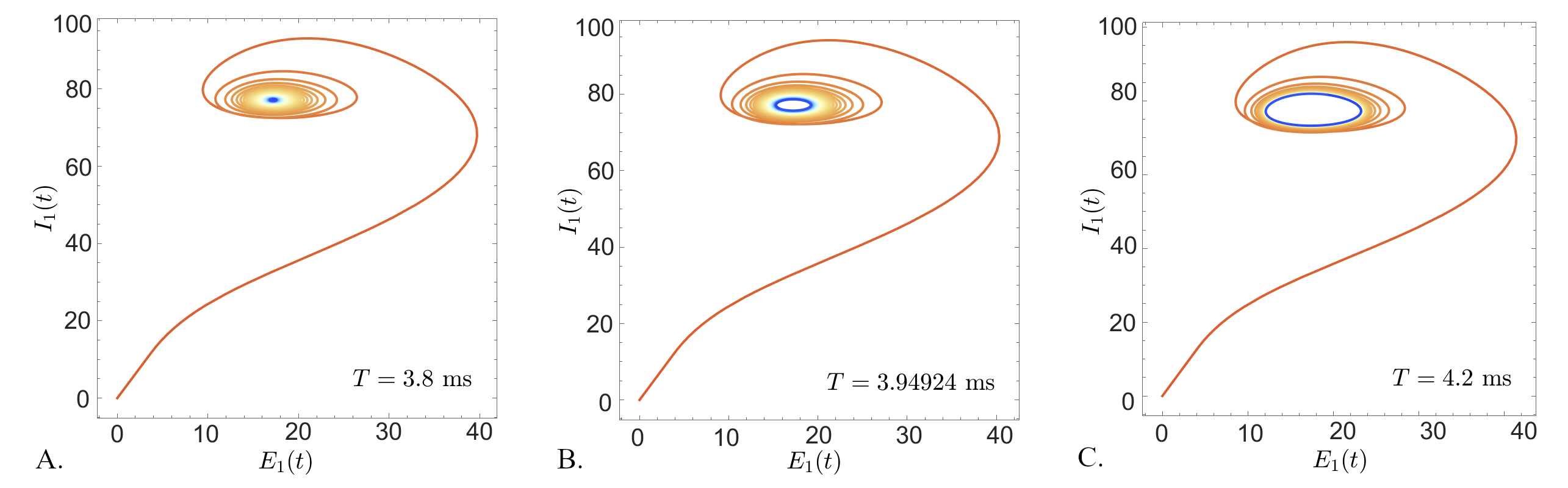}
\end{center}
\caption{\small \emph{{\bf Trajectory projection in the $(E_1,I_1)$ plane}, illustrating long-term behavior of the system with Dirac delays and fixed coupling parameters, for three different values of the delay, as follows: {\bf A.} for $T=3.8$ ms, the solution evolves towards a stable equlibrium; {\bf B.} at $T^*=3.94924$ ms, the system crosses a super-critical Hopf bifurcation, with formation of a stable cycle; {\bf C.} for $T=4.2$, the solution now settles to the stable cycle. All coupling strengths and other system parameters were fixed in all three cases to their table baseline values.}}
\label{simulation_Dirac}
\end{figure}

\begin{figure}[h!]
\begin{center}
\includegraphics[width=0.45\textwidth]{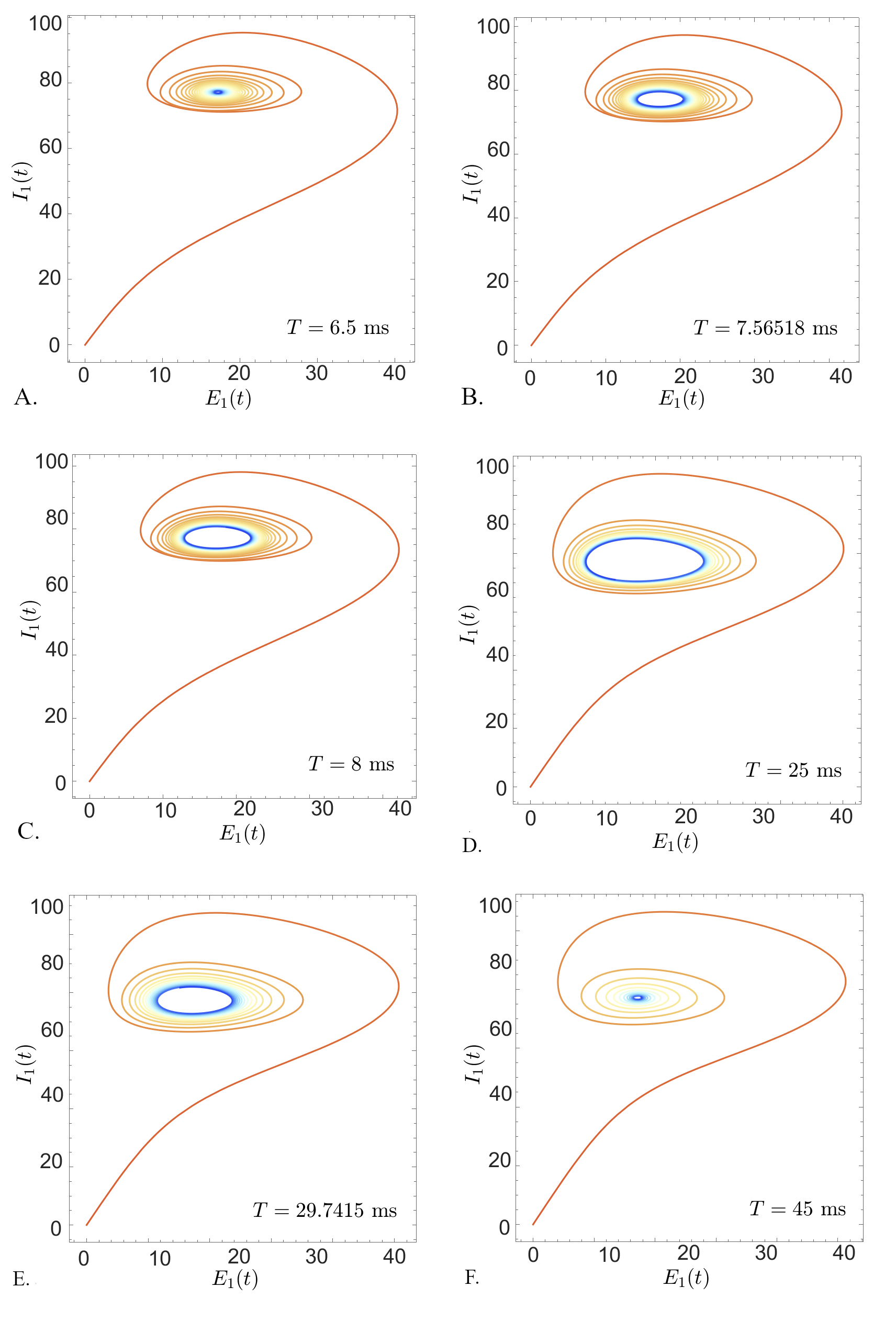}
\quad \quad
\includegraphics[width=0.45\textwidth]{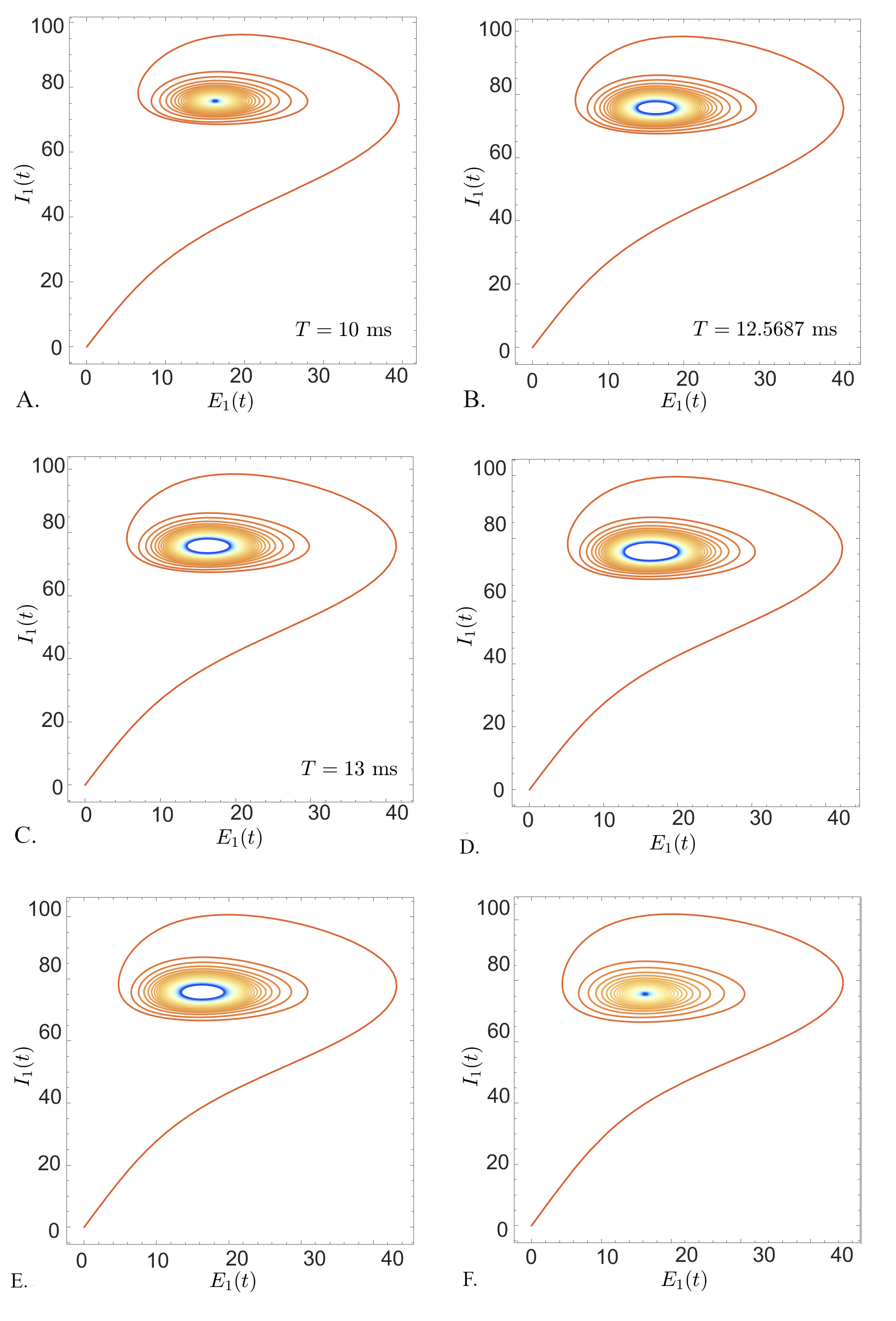}
\end{center}
\vspace{-5mm}
\caption{\small \emph{{\bf Trajectory projection in the $(E_1,I_1)$ plane}, illustrating long-term behavior of the system with weak Gamma distributed delays and fixed coupling parameters, for six different values of the delay. {\bf Left.} $w_{CS} = 6.6$, all other parameters fixed at table baseline values. The delay values are as follows: {\bf A.} $T=6.5$ ms, the solution evolves towards a stable equlibrium; {\bf B.} $T^{*}_{-}=7.56518$ ms, the system crosses a super-critical Hopf bifurcation, with formation of a stable cycle; {\bf C.} and {\bf D.} $T=8$ ms and $T=25$ ms, the solution settles to the stable cycle; {\bf E.} $T^*_{+}=29.7415$, the system goes through a subcritical Hopf bifurcation, with loss of oscillations; {\bf F.} $T=45$, the system again converges to a stable equilibrium. {\bf Right.} $W_{CS}=6.3$, all other parameters fixed at table baseline values. The delay values are as follows: {\bf A.} $T=10$ ms, the solution evolves towards a stable equlibrium; {\bf B.} $T^*_{-}=12.5687$ ms, the system crosses a super-critical Hopf bifurcation, with formation of a stable cycle; {\bf C.} and {\bf D.} $T=13$ ms and $T=15$ ms, the solution settles to the stable cycle; {\bf E.} $T^*_{+}=17.9016$, the system goes through a subcritical Hopf bifurcation, with loss of oscillations; {\bf F.} $T=25$, the system again converges to a stable equilibrium.}}
\label{simulation_weak}
\end{figure}

\subsubsection{Frequency analysis}

Finally, one of our main mathematical goals has been to understand how delays in processing inputs can underlie, through their length and their distribution, various mechanisms for triggering and stopping oscillations. Oscillations in firing rates are important in brain circuits in general, identifiable oscillatory rhythms being often associated with specific contexts or functional behaviors. This is true in particular for our circuit, since basal ganglia oscillations in the Beta range are considered a mark of dysregulation, and pathologies like Parkinson's Disease. We have thus far used our model to detect and analyze a few potential mechanisms for triggering and sustaining oscillations in the GPe-STN coupled unit. We found that such oscillations can be triggered by delays in input processing, with discrete delays being additionally oscillation prone versus weak Gamma distributed delays. Since many studies suggest that the functional and clinical importance of oscillations in this system are tied to the Beta range, we next examine whether our results remain true when we confine our analysis to entering and leaving the Beta band width. 

There are many different ways to illustrate the effects of changes in coupling parameters and changes in delay on oscillation frequency, once oscillations are triggered. We start with a simple illustration in Figure~\ref{WGS_WSC_freq_plane}, which shows (for each of the two delay kernels) the frequency band of the oscillations at their onset $T^{*}$, over the whole plane of coupling pairs $(W_{GS},W_{SC})$. The stability region in each case is shown in white, and the colors correspond to the different frequency bands, with the Beta range (12-30 Hz) in light orange. There are significant differences between the two panels. For the Dirac kernel, the frequency band at oscillation onset depends strongly on the point $(W_{GS},W_{SC})$ in the parameter plane: for values more central to the region of potential oscillations, where $T^{*}$ is the highest, the frequencies are also high ($>30$ Hz, corresponding to the Gamma band). Moving towards the boundary of the regions (i.e., closer to the stability region, or equivalently in the direction of decreasing $T^{*}$), the system has lower onset frequencies, traversing a large Beta range locus, then subsequently Alpha (8-12 Hz), Theta (4-8 Hz) and Delta ($<4$ Hz) bands. In contrast, for the weak Gamma distribution, the transition from the Beta band locus to the stability region is very abrupt, with the other bands hardly visible close to the boundary. 

This is interesting, especially in the context previously discussed of Dirac kernel promoting oscillations. For discrete delays, oscillations are possible  for a larger region in the $(W_{GS},W_{SC})$ plane, at lower delays and without a delay-based exit mechanism (allowing for persistent cycling with arbitrary large lags). This seemingly gives the weak Gamma the upper hand in precluding pathological behaviors. However, if our behavior of interest is Beta band oscillations, this advantage may have to be reassessed. 

\begin{figure}[h!]
\begin{center}
\includegraphics[width=0.95\textwidth]{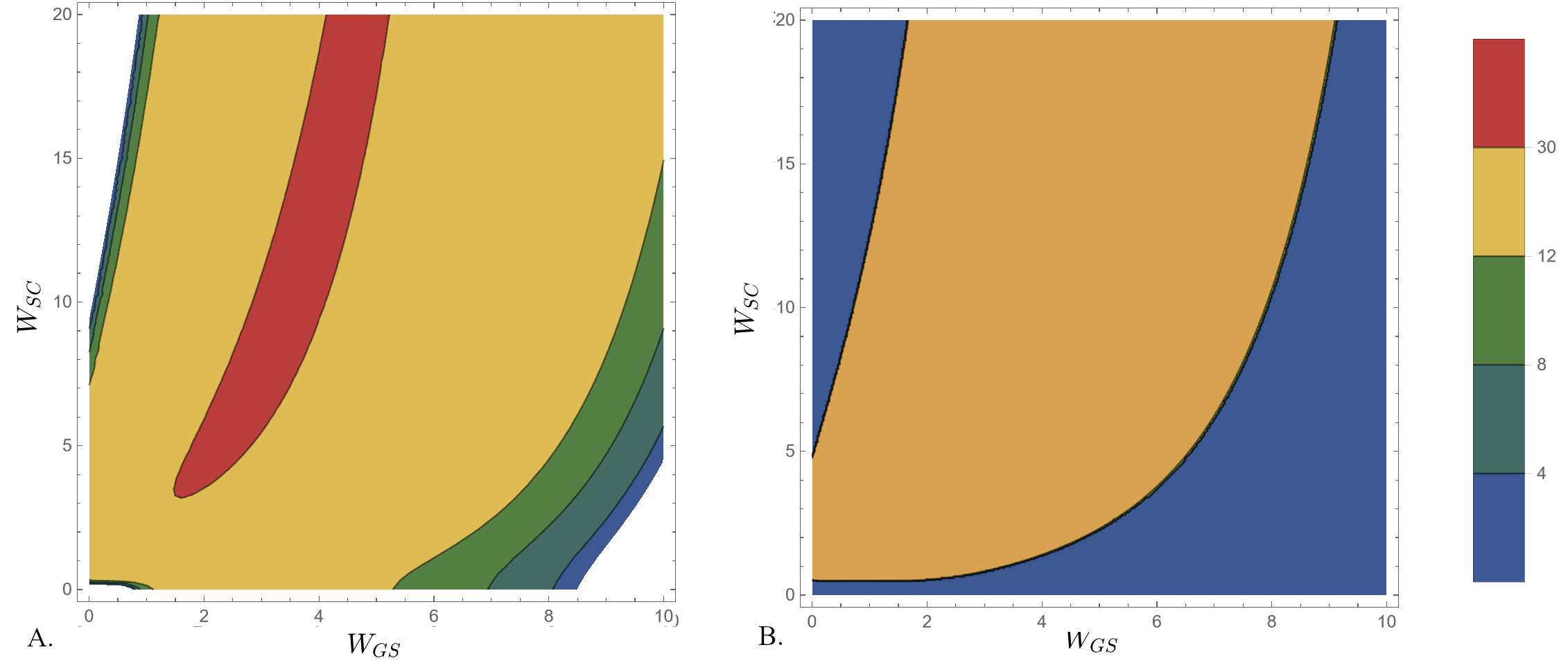}
\end{center}
\caption{\small \emph{{\bf Frequency bands at oscillation onset,} for the Dirac kernel ({\bf A}) and for the weak Gamma kernel ({\bf B}), shown in the $(W_{GS},W_{SC})$ parameter plane. In both panels, the respective stability region is shown in white, and the colors represent different frequency ranges, as specified in the color bar, from top to bottom: brown (gamma band, $>30$ Hz); orange (beta band, 12-30 Hz); green (alpha band, 8-12 Hz); dark green (theta band, 4-8 Hz); blue (delta band, $<4$ Hz).}}
\label{WGS_WSC_freq_plane}
\end{figure}

While it is still true that the Beta band locus is smaller for the weak Gamma case than for the Dirac case, Figure~\ref{WGS_WSC_freq_plane}b implies that this slight advantage comes with a downfall. Even small perturbations in coupling parameters (either $W_{SG}$ or $W_{SC}$) could push the system suddenly and without warning into Beta band oscillations, without transitioning through lower band regions. This could be clinically very important, in the context of being able to predict when the system may enter the pathological behavior associated with Beta oscillations. It generally supports our emphasis on the importance of obtaining information on the delay distribution type that the system is implementing.


\subsection{Modeling oscillations in a prefrontal-limbic emotion regulatory circuit}

In this section, we will illustrate, as proof of principle, how our theoretical results can be applied to understand the mechanisms behind dynamic patterns in two large classes of circuits, shown in Figure~\ref{PFC_BLA_circuits}. However, the hardwiring of the network, together with the signs assigned to the connection weights (excitatory vs. inhibitory) can identify each model with specific examples of local circuits in the brain, and deliver testable hypotheses that could apply specifically to these circuits. 

For example, to fix our ideas, we can view the two example circuits are two potential candidate models for representing the main interactions between the prefrontal cortex and the basolateral amygdala. These interactions are known to be extremely significant to mechanisms of fear extinction and emotion control, and are believed to be responsible for emotional dysregulations such as anxiety, bipolar disorder and schizophrenia. It is therefore very important to understand how circuitry and connectivity perturbations can affect the circuit's functional dynamics. 

It is generally accepted in the neuroscience literature that the PFC and the BLA are connected via long-range glutamate-mediated projections. Projection neurons in the BLA comunicate inputs to the PFC, and pyramidal cells in the PFC send feedback projections back to the BLA. One empirically-documented way in which these connections are made is via  abundant populations of inhibitory (GABA-ergic) interneurons, the function and connectivity of which  has been intensely studied in both PFC~\cite{lee2017computational} and BLA~\cite{lee2017amygdala,mcgarry2017prefrontal}. More specifically, excitatory projections from the BLA end up exercising an efficient inhibitory control over the pyramidal (PYR) neurons in the PFC, by activating the subpopulations of parvalbumin-expressing
(PV-) and somatostatin-expressing (SST-) interneurons (INs)~\cite{yang2021prefrontal} which in turn inhibit PYR activity. Conversely, the PYR glutamatergic projections to the BLA have in fact an indirect inhibitory effect on the BLA projection neurons, due to their relay through BLA inhibitory INs~\cite{lee2017amygdala,mcgarry2017prefrontal}. This hypothesized local circuit or PFC-BLA interactions, which we will call MODEL A, corresponds schematically with our wiring scheme in Figure~\ref{circuits}c, and is also represented in Figure~\ref{PFC_BLA_circuits}a, contextualized to match the neural circuitry described above.

A second alternative discussed more recently in the neuroscience literature is that of direct connections between projection neurons in PFC and BLA. For example, a dye empirical study identified strong axonal projections from the BLA to the prelimbic PFC, and an overlap of PFC axons connecting back to amygdala-cortical neurons in the BLA~\cite{mcgarry2017prefrontal}. This led to the standing hypothesis that PFC inputs selectively contact neurons in the BLA which conversely project directly back to PFC pyramidal neurons. We represented and analyzed this connectivity scheme in MODEL B, which corresponds to the general hardwiring geometry in Figure~\ref{circuits}a, and is further illustrated in Figure~\ref{PFC_BLA_circuits}b in its physiological context.

For the rest of the section, we will use numerical simulations in conjunction with our theoretical results, to illustrate potential behaviors of MODEL A and B, and differences between them. We will focus in particular on identifying oscillation-triggering and stopping mechanisms, the role of which is known to be crucial in prefrontal-limbic dynamics. Gamma oscillations have been extensively studied in the BLA recently~\cite{schonfeld2019beyond,headley2021gamma},and are believed to be significant to emotional regulation and social behavior~\cite{kuga2022prefrontal}. It is not yet know whether the two schemes described in MODEL A and MODEL B act simultaneously, or if the system rather has to choose between one versus the other in order to deliver a certain behavior or function. In this section, we study them independently. Further coupling does not fall within the schematic options discussed in our general setup, and will not be discussed in this paper.

For simulations in both models, the integrating sigmoidal functions were chosen to be those in the original reference by Wilson and Cowan~\cite{wilson1972excitatory}, of the form
$${\cal F}(u) = \frac{1}{1+\exp{-b(u-\theta)}} - \frac{1}{1+\exp{b \theta}}$$
where the sigmoidal parameters for the excitatory nodes were $b_e=1.2$ and $\theta_e =4$, and for the inhibitory nodes were $b_i=1$ and $\theta_i=2$. The coupling parameters used for the numerical illustrations were also loosely based on the values and ranges in the same reference, and are specified in each simulation.

\begin{figure}[h!]
\begin{center}
\includegraphics[width=0.8\textwidth]{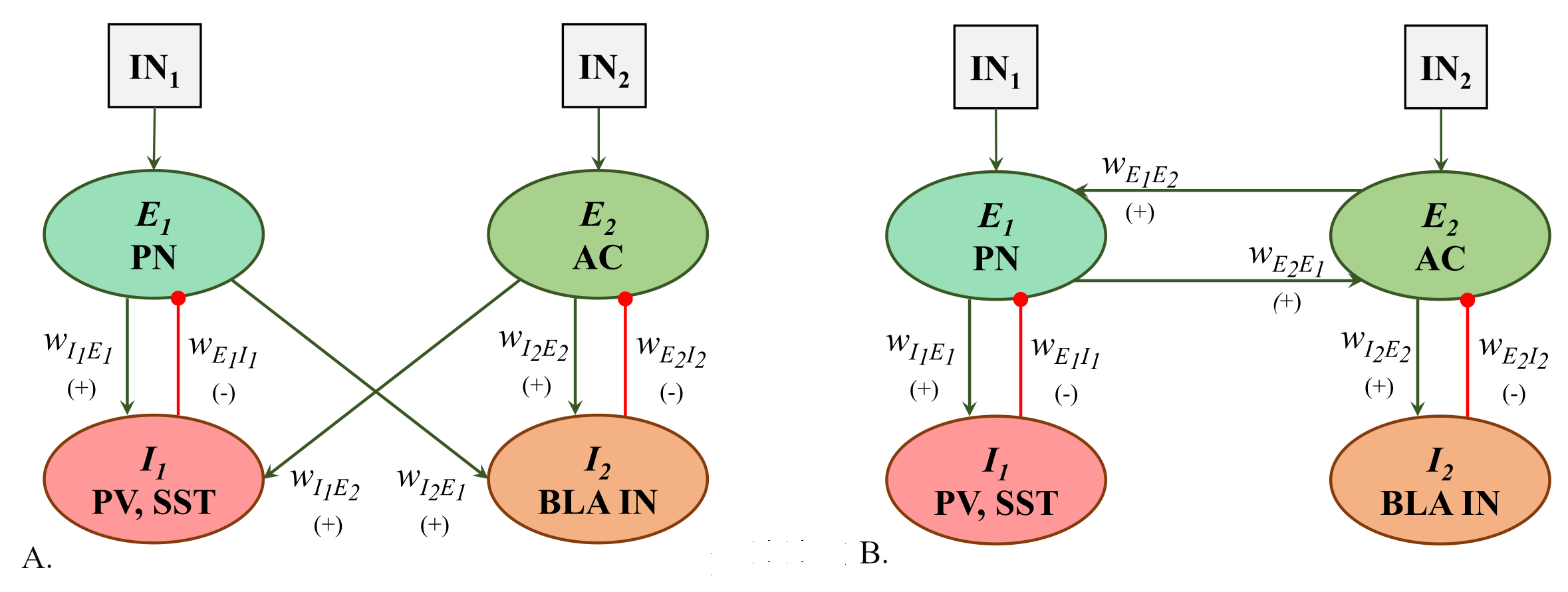}
\end{center}
\caption{\small \emph{{\bf Simplified prefrontal-amygdala coupling circuitry} as described in MODEL A (left) and MODEL B (right). The nodes/neural populations are in each case $PN$ (prefrontal pyramidal neurons, designated as the first excitatory population $E_1$; PV and SST (parvalbumin- and somatostatin-expressing  interneurons, representing the first inhbitory population $I_2$); AC (amygdala-cortical, long range projection neurons, representing the second excitatory population $E_2$); $BLA IN$ (amygdala gabaergic interneurons, the second inhibitory populations $I_2$). The weights and signatures of both local and long-range projections are represented by arrows (green for exicitatory, and red for inhibitory connections). External inputs IN$_1$ and IN$_2$ are received by the excitatory units $E_1$ and $E_2$, respectively.}}
\label{PFC_BLA_circuits}
\end{figure}

In both MODEL A and MODEL B, we considered the same E/I interactions (with the same coupling parameters) \emph{within} the BLA and the PFC. We assumed strong inhibition $\ds w_{_{E_1I_1}} =w_{_{E_2I_2}}$ from the GABAergic interneurons to the corresponding projection neurons. We also included excitatory feedback from the pyramidal cells to the PV and SST INs, as well as from the amygdala-cortical cells to the BLA INs. While such connections have been evidenced empirically, they don't seem to bring as strong a contribution to the functional circuitry in question, hence we assigned them relatively low baseline values $\ds w_{_{I_1E_1}}= w_{_{I_2E_2}}$ compared to the magnitude of the feed-forward inhibition. The key parameters in both models were taken to be the strengths of the glutamatergic projections between PFC and BLA. More precisely, $w_{_{I_2E_1}}$ and $w_{_{I_1E_2}}$ (in MODEL A), and $w_{_{E_2E_1}}$ and $w_{_{E_1E_2}}$ (in MODEL B) were varied for our analysis within a wide range of values. The specific values and ranges are provided in Table~\ref{PFC_BLA_params}.

 \renewcommand{\arraystretch}{1.5}
\begin{table}[h!]
\begin{center}
\begin{tabular}{|l|l|l|l|}
\hline
\multicolumn{2}{|c|}{\bf MODEL A} & \multicolumn{2}{|c|}{\bf MODEL B}\\
\hline
Parameter & Value/range & Parameter & Value/range\\
\hline 
$w_{_{I_1E_1}}$, $w_{_{I_2E_2}}$ & 2 & $w_{_{I_1E_1}}$, $w_{_{I_2E_2}}$ & 2 \\
$w_{_{E_1I_1}}$, $w_{_{E_2I_2}}$ & -16 & $w_{_{E_1I_1}}$, $w_{_{E_2I_2}}$  & -16 \\
$w_{_{I_2E_1}}$, $w_{_{E_2I_1}}$ & [0,10] & $w_{_{E_1E_2}}$, $w_{_{E_2E_1}}$ & [0,10] \\
IN$_{1,2}$ & 6 & IN$_{1,2}$ & 6 \\
\hline
\end{tabular}
\end{center}
\caption{\small \emph{{\bf Coupling parameter values and ranges for MODEL A and MODEL B.} }}
\label{PFC_BLA_params}
\end{table}

Figure~\ref{PFC_BLA_curves} illustrates the behavior of the system as one of the cross-coupling parameters is held fixed, while the other is being changed. Notice that, although the parameters are moving along the same interval, the curves described in the $(\alpha,\beta)$ plane are significantly different between models (due to the fact that $\alpha$ and $\beta$ are different functions with respect to the coupling weights between MODEL A and MODEL B, as shown in Table~\ref{table_coefficients}). In fact, the paths lie in widely different regions of the $(\alpha,\beta)$ domain (underneath versus above the parabola $\ds \beta = \frac{\alpha^2}{4}$), distinctly outlining behaviors and transitions into and out of oscillations triggered by different mechanisms. 

These behaviors and transitions are illustrated separately in Figure~\ref{PFC_BLA_curves} for the Dirac and the weak Gamma kernel (left versus right panels). In contrast with Figure~\ref{first_crit_delay}, here we show the critical value $\ttau^*$ over the plane $(\alpha,\beta)$ through a set of level curves, the colors representing the corresponding levels (as shown in the color bar). This way, it is easier to observe how a curve in $(\alpha,\beta)$ plane evolves through higher and then lower values of $\ttau^*$. In each case, we track two parameter trajectories obtained by fixing the strength of one of the cross-connections ($w_{_{E_2I1}}$ and $w_{_{E_2E_1}}$, respectively) and varying the other cross-connection strength ($w_{_{E_1I2}}$ and $w_{_{E_1E_2}}$, respectively). Notice that the trajectories for MODEL A (purple curves) evolve in a different region of the $(\alpha,\beta)$ plane (below the parabola $\beta = \frac{\alpha^2}{4}$, not shown) than the trajectories for MODEL B (blue curve), which lie above the parabola. 

In the case of Dirac kernel, the trajectories spend most of their course outside of the stability region. Increasing on of the cross-connectivity value takes both models along a unimodal path of gradually lower, and then increasingly higher critical delay values $\ttau^*$. For a Dirac system with fixed delay $\ttau$, this results in crossing in and out of oscillations, as the curve traverses the $\ttau=\ttau^*$ level twice in its course, with a more pronounced interval of oscillations for thick trajectories (where the fixed cross-connectivity is higher). This reflects the fact that, although the effect is similar in both models (in that oscillations are triggered and then stopped), this is accomplished through different mechanisms between the two models, involving different regions in the $(\alpha,\beta)$ plane, with different computational scheme for the critical $\ttau^*$ where oscillations appear.

\begin{figure}[h!]
\begin{center}
\includegraphics[width=\textwidth]{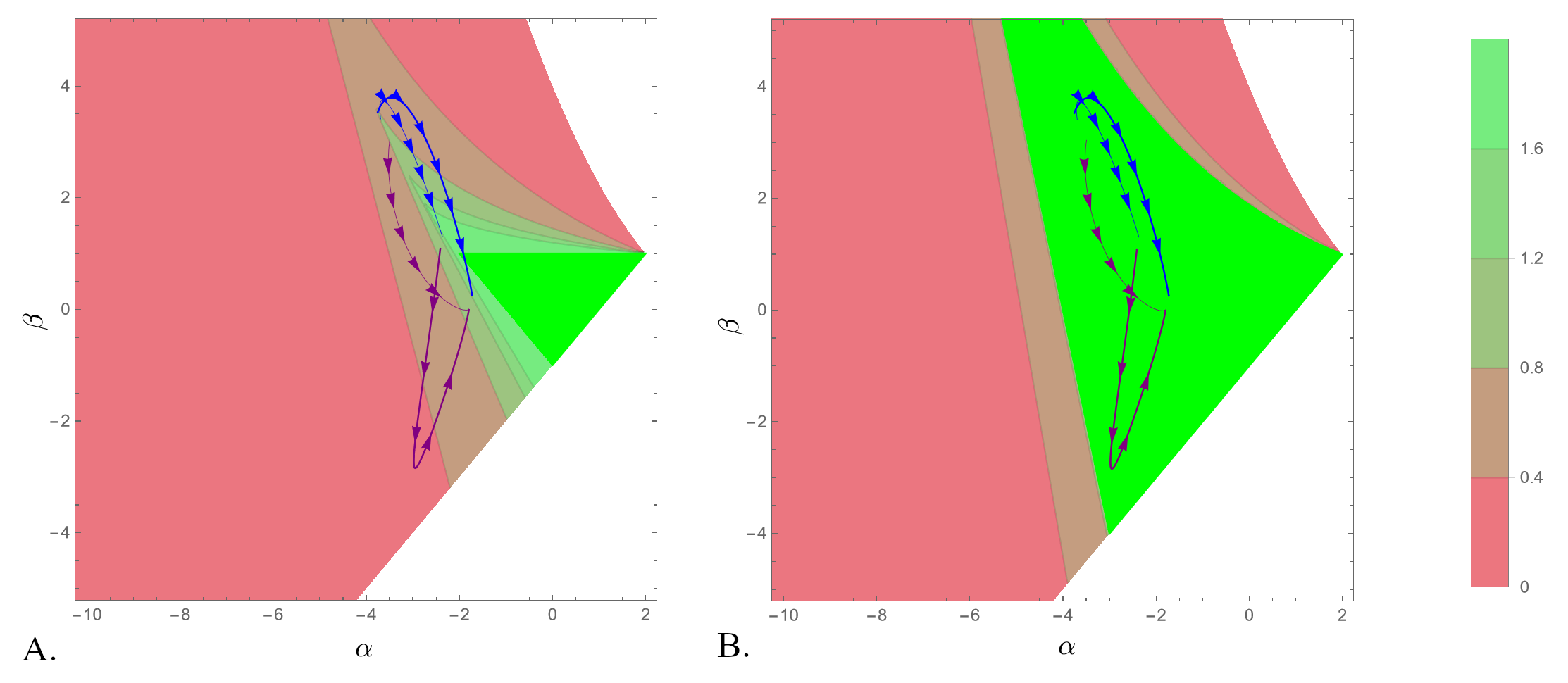}
\end{center}
 \caption{\small \emph{{\bf Trajectories in the $(\alpha,\beta)$ plane when one coupling weight is fixed and the other one is increased.} For MODEL A (purple curves), the thin curve represents $w_{_{E_2I_1}}=1$, and the thick curve represents $w_{_{E_2I_1}}=3$, with $w_{_{E_1I_2}}$ increasing in $[0,10]$. For MODEL B (blue curves), the thin curve represents $w_{_{E_2E_1}}=1$, and the thick curve represents $w_{_{E_2E_1}}=3$, with $w_{_{E_1E_2}}$ increasing in $[0,10]$.
 {\bf A.} The stability region and sample level curves for the Dirac kernel are shown (colored according to the bar on the right), to illustrate how the trajectories move across different values for the critical $\ttau*$ in the case of discrete delays. {\bf B.} In the corresponding picture for the weak Gamma kernel, the trajectories do not leave the stability region. }}
\label{PFC_BLA_curves}
\end{figure}

To better contextualize how the discrete delay value together with the values of the cross-connectivities, govern the behavior of the system -- we include two additional figures, in Appendices 2 and 3. The figures illustrate how the trajectory of the system from rest (initial condition $(E_1,I_1,E_2,I_2)=(0,0,0,0)$) changes as the delay $\ttau$ is increased, for different values of the connectivity parameters. The trajectories are represented as their projection in the $(E_1,I_1)$ slice of the four-dimensional parameter space. While in each case a cycle forms at the critical $\ttau^*$ corresponding to the respective parameter set, the position of the cycle, its geometry and kinetics are significantly different between the two models. This not only reiterates the importance of the connectivity pattern (hardwired structure) over connectivity weights (adjustable strengths), in determining the dynamic fate of the systems. It also suggests that an adequate mathematical perspective (e.g., identifying fundamentally different connectivity structures with different loci in the $(\alpha,\beta)$ plane) can help understand and predict the differences between the subsequent behaviors. 

\section{Discussion}

In this paper, we modeled and studied long-term dynamics in systems of coupled pairs of Wilson-Cowan oscillators with distributed delays, under different connectivity schemes, for different connection strengths, different delay distributions and average delay values. First, we worked on deriving general (i.e., structure and kernel independent) results. We then focused on understanding the particular contribution of each of these aspects to the system's dynamics and on how changes in these factors reflect into the system's behaviors and transitions.

Altogether, our analysis of this four-dimensional system of coupled Wilson-Cowan pairs supports the idea that the presence and type of distributed delays in input processing are a strong factor determining the system's long-term behavior. We had considered this idea in our prior work on the dynamics of a single Wilson-Cowan pair, showing how the choice of delay and of connection weights affect the system's transitions in and out of oscillations. The current context warrants a finer discussion, since the system under consideration is richer in terms of the size, number of parameters and coupling schemes. While this is advantageous from an applications perspective, allowing us to investigate behavior in more realistic brain circuits, it also requires a more complex mathematical analysis, that can consider the additional impact of configuration in conjunction with the contributions of the delays and coupling strengths. 

For the present paper, to fix our ideas, we considered four connectivity schemes simple enough to allow for convenient mathematical reductions (e.g., in the form of the characteristic equation of the linearized system), but rich enough to be realistically used to model important neuro-regulatory systems. As working examples, we applied our theoretical analysis to a basal ganglia circuit, in which Beta oscillations have been identified as a signature of PD symptoms, and to a prefrontal-limbic regulatory system, to illustrate potential mechanisms behind the Gamma oscillations measured in the basolateral amygdala.


\section*{Appendix 1}

{\bf Case 1.} Suppose we are at a point in the parameter region 
\begin{equation}
\frac{\alpha^2}{4} < \beta < \frac{(\alpha-4)^2}{4}, \; \alpha<2
\label{above_parabola}
\end{equation}
Then, oscillations can only be triggered through a Hopf bifurcation if equation~\eqref{omega_case1} has solutions. In the case of the weak Gamma kernel, $\ds \rho(\omega) = \frac{1}{\sqrt{\omega^2+1}}$ and $\theta(\omega) = \arctan \omega$, hence the equation becomes:
\begin{equation}
    \alpha = 2\sqrt{\beta} \cos \left[ 2 \left( \arctan \sqrt{\ds \frac{\sqrt{\beta}}{\omega^2+1}-1} + \arctan \omega \right) \right]
    \label{alpha_above_parabola}
\end{equation}

Using the trigonometric identity $\ds \arctan x + \arctan y = \arctan \frac{x+y}{1-xy}$, then the double angle formula $\cos(2x) = 2\cos^2 x-1$ in conjunction with $\ds \cos (\arctan x) = \frac{1}{\sqrt{1+x^2}}$, this can be simplified as:
\begin{equation}
    \alpha = f(\omega) = - \left[ \sqrt{\beta} \left( -2 + \frac{4}{1+\omega^2} \right) + 4\left( -1+\omega^2+2\omega\sqrt{\ds -1+\frac{\sqrt{\beta}}{1+\omega^2}} \right) \right]
    \label{alpha_above}
    \end{equation}

Clearly, $\alpha = f(\omega)$ has solutions only if $\alpha$ is in the range of $f(\omega)$, which we will determine by carrying out a classical optimization computation. First, notice that the existence of the radical restricts the domain of $f(\omega)$ to $0 \leq \omega \leq \omega_{\max} = \sqrt{\sqrt{\beta}-1}$, which is nontrivial only if $\beta>1$ (which is in fact in agreement with the conditions in Theorem~\ref{stability_theorem}). The values of $f$ at the endpoints of this interval are:
$$f(0) = f(\omega_{\max}) = 4-2\sqrt{\beta}$$

We then need to find the critical points of $f(\omega)$. Three critical points are theoretically possible:
\begin{eqnarray*}
u_{1,3} &=& \frac{-1+\sqrt{\beta}\mp \sqrt{-4\sqrt{\beta}+\beta}}{\sqrt{2}}\\
u_2 &=& \sqrt{-1+\sqrt[4]{\beta}}
\end{eqnarray*}

A symbolic analysis and classification of this points, distinguished the following cases:\\

\noindent {\bf Case A: $1<\beta<16$.} Then $u_2$ is the unique critical point, and it is a local minimum, with $f(u_2) = 2(8-8\sqrt[4]{\beta}+\sqrt{\beta}) < f(0)$. Hence $f(u_2) < \alpha < f(0)$. In conjunctions with the original conditions~\ref{above_parabola} and the restriction on $\beta$ in this case, we obtain the parameter region described by:
\begin{equation*}
1<\beta<16,\quad 2(8-8\sqrt[4]{\beta}+\sqrt{\beta}) < \alpha < 4-2\sqrt{\beta}
    \label{alpha_case1}
\end{equation*}

\noindent {\bf Case B: $16<\beta<81$}. In this case, all three critical points are real, and $f(0)>f(u_2)>f(u_{1,3})$, hence $f(u_1) < \alpha < f(0)$. Together with the other conditions, this leads to the parameter region:
\begin{equation*}
16<\beta<81,\quad 0 < \alpha < 4-2\sqrt{\beta}
    \label{alpha_case2}
\end{equation*}

\noindent {\bf Case C. $\beta>81$.} The condition makes $u_2$ a local maximum, and $u_{1,3}$ local minima. Then $f(u_1)<\alpha<f(u_2)$. Together with the other conditions, we get:
\begin{equation*}
\beta>81,\quad 0 < \alpha < 4-2\sqrt{\beta}
    \label{alpha_case3}
\end{equation*}

Taken together, these cases imply that the equation~\eqref{alpha_above_parabola} has no solutions under the conditions~\ref{above_parabola} if 
\begin{equation}
1<\beta<16, \quad -2\sqrt{\beta} < \alpha < 2(8-8\sqrt[4]{\beta}+\sqrt{\beta})
\label{region_above_parabola}
\end{equation}

Above the parabola and outside of the region~\eqref{region_above_parabola}, the equation in $\omega$~\eqref{alpha_above} can have either two or four real roots. More specifically, between the curves $\alpha = 2(8-8\sqrt[4]{\beta}+\sqrt{\beta})$ and $\ds \beta = \frac{(\alpha-4)^2}{4}$ (gray region in Figure~\ref{regions}), equation~\eqref{alpha_above} has only two positive roots 
\begin{eqnarray}
    \omega_{\mp} &=& \frac{\sqrt{-8-\alpha \mp \sqrt{\left( \ \alpha-2\sqrt{\beta} \right) \left( 16+\alpha -8\sqrt{\alpha+2\sqrt{\beta}} -2\sqrt{\beta} \right)} + 4\sqrt{\alpha+2\sqrt{\beta}} + 2\sqrt{\beta}}}{2\sqrt{2}}
    \label{gray_region}
\end{eqnarray}

with so that $\omega_{-}<\sqrt{-1+\sqrt[4]{\beta}}<\omega_{+}$. In the region between $\ds \beta = \frac{\alpha^2}{4}$ and $\ds \alpha = 2(8-8\sqrt[4]{\beta}+\sqrt{\beta})$ (shown in pink in Figure~\ref{regions}), the equation has two additional positive roots, given by 
\begin{eqnarray}
    v_{\mp} &=& \frac{\sqrt{-8-\alpha \mp \sqrt{\left( \ \alpha-2\sqrt{\beta} \right) \left( 16+\alpha +8\sqrt{\alpha+2\sqrt{\beta}} -2\sqrt{\beta} \right)} - 4\sqrt{\alpha+2\sqrt{\beta}} + 2\sqrt{\beta} }}{2\sqrt{2}}
    \label{pink_region}
\end{eqnarray}

so that $\omega_{-}<v_{-}<\sqrt{-1+\sqrt[4]{\beta}}<v_{+}<\omega_{+}$.

\begin{figure}[h!]
\begin{center}
\includegraphics[width=0.6\textwidth]{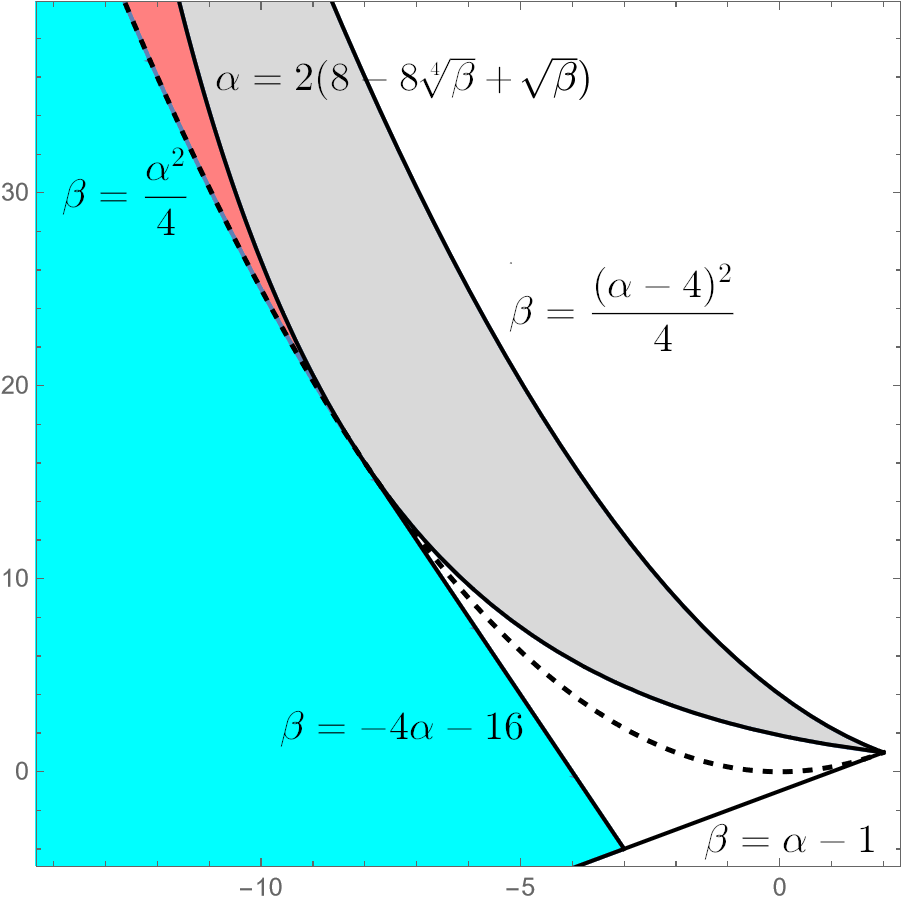}
\end{center}
\caption{\small \emph{{Regions in the $(\alpha,\beta)$ parameter plane} corresponding to different successions and behaviors of critical delays. In the grey region and cyan regions, there are two critical delays $\ttau^*(\omega_{-})<\ttau^*(\omega_{+})$, corresponding to the $\omega_{\pm}$ obtained in equations~\eqref{gray_region} and equations~\eqref{cyan_region}, respectively. In the pink regions, there are four critical delays, with the lowest values $\ttau^*(\omega_{-})$ provided by equation~\eqref{pink_region}.}}
\label{regions}
\end{figure}

\vspace{2mm}
{\bf Case 2.} Suppose now that we are at a point in the parameter region 
\begin{equation}
\beta < \frac{\alpha^2}{4}, \: \alpha<2
\label{below_parabola}
\end{equation}

Here, oscillations can only be triggers when equation~\eqref{omega_case2} has solutions. In the case of the weak Gamma kernel, this equation is:
\begin{equation*}
    \frac{\omega}{\omega^2+1} = \frac{2}{\lvert \alpha - \sqrt{\alpha^2-4\beta}\rvert}
\end{equation*}

This equation has no real roots when $-8 + \lvert \alpha - \sqrt{\alpha-4\beta} \rvert < 0$. In this case, there is no critical $\tau^*$ that can trigger oscillations, hence the equilibrium is stable for all delays. Together with condition~\eqref{below_parabola}, this is the parameter region 
\begin{equation}
-8 < \alpha <2, \: -16- 4\alpha < \beta < \frac{\alpha^2}{4}
\label{region_below_parabola}
\end{equation}

Below the parabola~\eqref{below_parabola} and outside of the region~\eqref{region_below_parabola}, the equation has two roots:
\begin{equation}
\omega_{\pm} = \frac{\pm\sqrt{-8 + \lvert \alpha - \sqrt{\alpha^2-4\beta}\rvert} + \sqrt{\lvert \alpha - \sqrt{\alpha^2-4\beta}\rvert}}{2\sqrt{2}}
\label{cyan_region}
\end{equation}
In the region specified, $0< \omega_{-} < 1 < \omega_{+}$. 

In particular, this computation provides us with a stability parameter region for the case of the weak Gamma kernel (emerging from joining conditions~\eqref{region_above_parabola} and~\eqref{region_below_parabola} that extends substantially over the original, kernel independent, stability locus, and that can be described as the region enclosed by 
 the line $\beta = -4\alpha-16$, the parabola $\alpha = 2 \left( 8 - 8 \sqrt[4]{\beta} + \sqrt{\beta} \right)$ and the line $\beta = \alpha -1$.

 \begin{landscape}
 \section*{Appendix 2A}

\begin{figure}[h!]
\begin{center}
\includegraphics[width=0.9\linewidth]{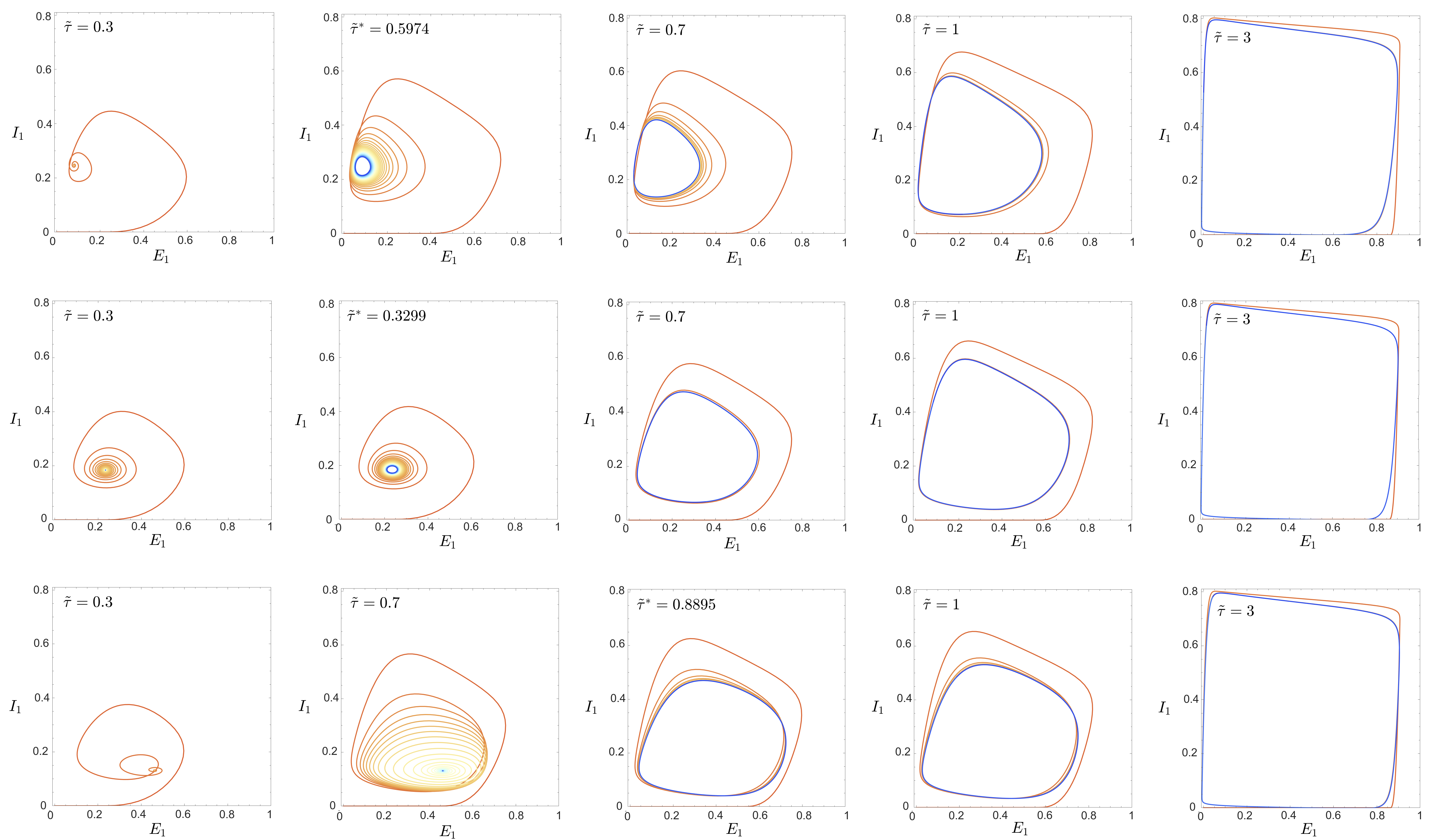}
\end{center}
\caption{\small \emph{{\bf Trajectory projection in the $(E_1,I_1)$ plane, illustrating the long-term behavior of MODEL A with Dirac discrete delays}, for different balances of $w_{_{E_1I_2}}$ versus $w_{_{E_2I_1}}$, and different delay values. For all simulations, the former was fixed to $w_{_{E_1I_2}}=3$, and the latter was changed from $w_{_{E_2I_1}}=1$ ({\bf top} row) to $w_{_{E_2I_1}}=3$ ({\bf middle} row) to $w_{_{E_2I_1}}=5$ ({\bf bottom} row). For each row, five delay values $\ttau$ were explored: $\ttau=0.3$, $\ttau=\ttau^*$ (computed based on the corresponding case), $\ttau=0.7$, $\ttau=1$ and $\ttau=3$. This illustrates the formation of the stable cycle at $\ttau^*$ and its subsequent evolution.}}
\label{ModelA_sim}
\end{figure}

\clearpage
 \section*{Appendix 2B}
 
\begin{figure}[h!]
\begin{center}
\includegraphics[width=0.9\linewidth]{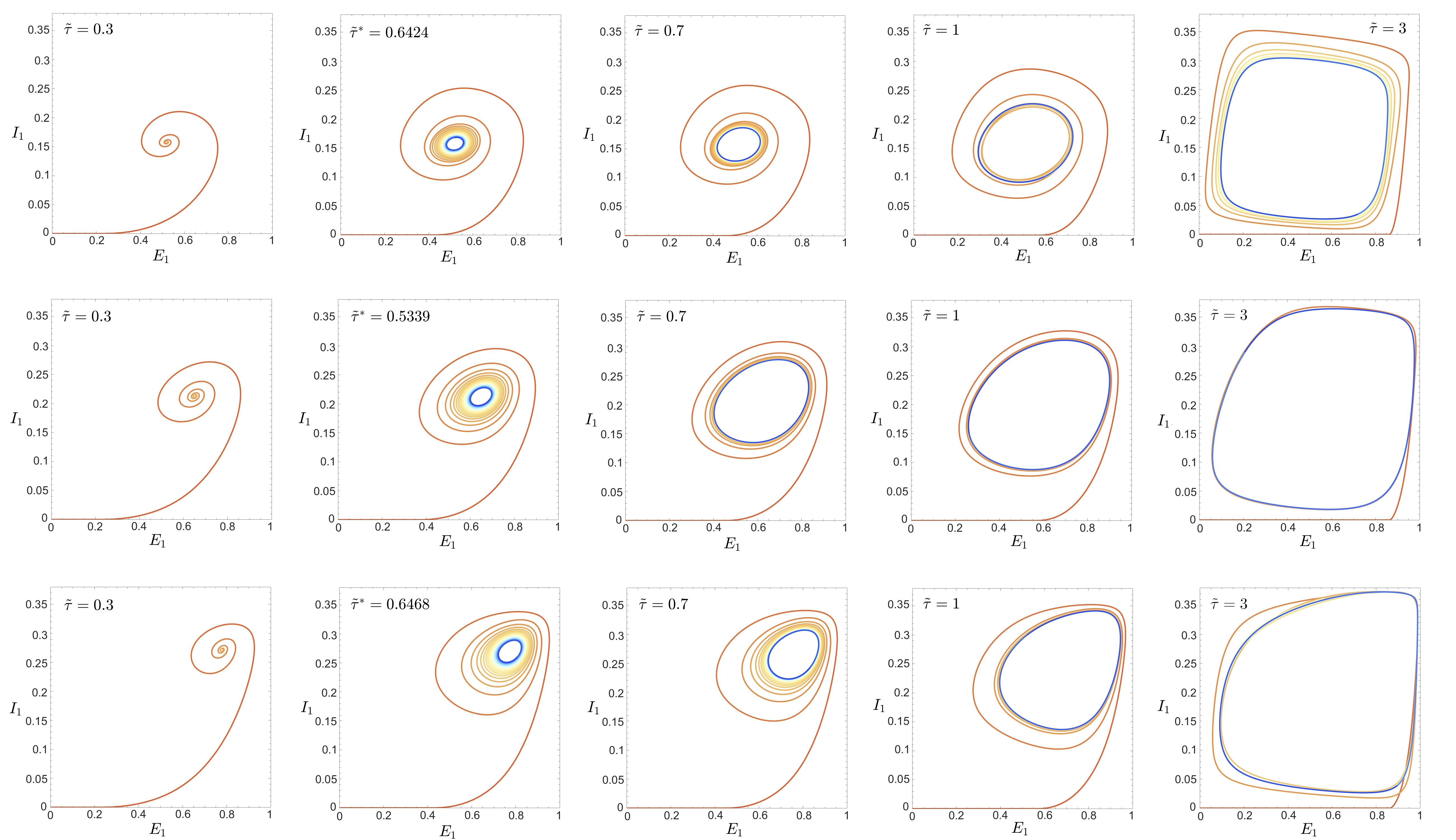}
\end{center}
\caption{\small \emph{{\bf Trajectory projection in the $(E_1,I_1)$ plane, illustrating the long-term behavior of MODEL B with Dirac discrete delays}, for different balances of $w_{_{E_1E_2}}$ versus $w_{_{E_2E_1}}$, and different delay values. For all simulations, the former was fixed to $w_{_{E_1E_2}}=3$, and the latter was changed from $w_{_{E_2E_1}}=1$ ({\bf top} row) to $w_{_{E_2E_1}}=3$ ({\bf middle} row) to $w_{_{E_2E_1}}=5$ ({\bf bottom} row). For each row, five delay values $\ttau$ were explored: $\ttau=0.3$, $\ttau=\ttau^*$ (computed based on the corresponding case), $\ttau=0.7$, $\ttau=1$ and $\ttau=3$. This illustrates the formation of the stable cycle at $\ttau^*$ and its subsequent evolution.}}
\label{ModelB_sim}
\end{figure}
     
 \end{landscape}

\bibliography{distribdelays}

\end{document}